\documentclass[a4paper, 10pt]{amsart}
\usepackage{amsmath}
\usepackage{amssymb, color, hyperref}

\theoremstyle{plain}
\newtheorem{theorem}{\bf Theorem}[section]
\newtheorem{proposition}[theorem]{\bf Proposition}
\newtheorem{lemma}[theorem]{\bf Lemma}
\newtheorem{corollary}[theorem]{\bf Corollary}

\theoremstyle{definition}

\newcommand{\N}{\mathbb N}
\newcommand{\Z}{\mathbb Z}

 \DeclareMathOperator{\ord}{ord}

 \DeclareMathOperator{\supp}{supp}

\renewcommand{\t}{\, | \,}

\numberwithin{equation}{section}

\begin{document}

\title{On the incomparability of systems of sets of lengths}

\address{Institute for Mathematics and Scientific Computing\\ University of Graz, NAWI Graz\\ Heinrichstra{\ss}e 36\\ 8010 Graz, Austria}
\email{alfred.geroldinger@uni-graz.at}
\urladdr{https://imsc.uni-graz.at/geroldinger}

\address{Laboratoire Analyse, G{\'e}om{\'e}trie et Applications, LAGA, Universit{\'e} Sorbonne Paris Nord, CNRS, UMR 7539, F-93430, Villetaneuse, France
 \\ and \\ Laboratoire Analyse, G{\'e}om{\'e}trie et Applications (LAGA, UMR 7539) \\ COMUE  Universit{\'e} Paris Lumi{\`e}res \\  Universit{\'e} Paris 8, CNRS \\  93526 Saint-Denis cedex, France} \email{schmid@math.univ-paris13.fr}

\author{Alfred Geroldinger  and Wolfgang A. Schmid}

\thanks{This work was supported by the Austrian Science Fund FWF, Project Number P33499.}



\begin{abstract}
Let $H$ be a Krull monoid with finite class group $G$ such that every class contains a prime divisor. We consider the system $\mathcal L (H)$ of all sets of lengths of $H$ and study when $\mathcal L (H)$ contains or is contained in a system $\mathcal L (H')$ of a Krull monoid $H'$ with finite class group $G'$, prime divisors in all classes and Davenport constant $\mathsf D (G')=\mathsf D (G)$. Among others, we show that if $G$ is either cyclic of order $m \ge 7$ or an elementary $2$-group of rank $m-1 \ge 6$, and $G'$ is any group which is non-isomorphic to $G$ but with Davenport constant $\mathsf D (G')=\mathsf D (G)$, then the systems $\mathcal L (H)$ and $\mathcal L (H')$ are incomparable.
\end{abstract}

\maketitle


\section{Introduction} \label{1}

Let $H$ be a Krull monoid or a Krull domain with class group $G$ such that every class contains a prime divisor. The system $\mathcal L (H)$ of all sets of lengths of $H$ is a well-studied invariant describing  factorizations in $H$. It is classic that  $H$ is factorial if and only if $|G|=1$ and that $H$ is half-factorial (i.e., $|L|=1$ for all $L \in \mathcal L (H)$) if and only if $|G|\le 2$.  All sets of lengths $L \in \mathcal L (H)$ are finite and, if $|G| \ge 3$, then for every $N \in \N$ there is  $L_N \in \mathcal L (H)$ such that $|L_N|\ge N$. Every finite subset of $\N_{\ge 2}$ lies in $\mathcal L (H)$ if and only if the class group $G$ is infinite. Suppose that $G$ is finite with $|G| \ge 3$. Then sets of lengths in $\mathcal L (H)$ are well-structured and depend only on the class group $G$. More precisely,  we have $\mathcal L (H) = \mathcal L (G) := \mathcal L \big( \mathcal B (G) \big)$, where $\mathcal B (G)$ denotes the monoid of zero-sum sequences over   $G$. We refer to \cite{Ge-HK06a, Ge-Ru09} for background on Krull monoids and their connection to additive combinatorics and to \cite{Sm13a, Fr-Na-Ri19a, Fa-Tr18a, Go19a, Tr19a} for recent work on  sets of lengths.

The standing conjecture is that the system $\mathcal L (H)$ determines the class group $G$ (apart from two trivial exceptions, given in Theorem A below).  This means that, if $H'$ is a Krull monoid with class group $G'$ such that every class contains a prime divisor, then $\mathcal L (H) = \mathcal L (H')$ if and only if the class groups $G$ and $G'$ are isomorphic. For small groups it is easy to write down their systems of sets of lengths. We denote by $C_n$ a cyclic group of order $n \in \N$ and recall the following well-known result (\cite[Proposition 3.3 and Theorem 3.6]{Ge-Sc-Zh17b}).

\medskip
\noindent
{\bf Theorem A.}

\begin{enumerate}
\item $\big\{ \{k\} \colon k \in \N_0 \big\}    = \mathcal L (C_1)=\mathcal L (C_2)$.

\item $\bigl\{ y
      + 2k + [0, k] \, \colon \, y,\, k \in \N_0 \bigr\} = \mathcal L (C_3)= \mathcal L (C_2 \oplus C_2) \subsetneq \mathcal L (G)$ for all finite abelian groups $G$ with $\mathsf D (G) \ge 4$.
\end{enumerate}

The above result covers all groups $G$ with Davenport constant $\mathsf D (G) \le 3$, and the above mentioned conjecture expects an affirmative answer to the following problem.

\medskip
\noindent
{\bf The Characterization Problem.} Let $G$ be a finite abelian group with Davenport constant $\mathsf D (G) \ge 4$, and let $G'$ be an abelian group with $\mathcal L (G) = \mathcal L (G')$. Are $G$ and $G'$ isomorphic?

\medskip
We refer to the surveys (\cite{ Ge16c, Ge-Zh20a}) for background on the Characterization Problem and to \cite{B-G-G-P13a, Zh18a, Zh18b, Ge-Sc19a, Ge-Zh17b} for recent progress. In \cite{Ge-Sc-Zh17b}, the Characterization Problem was studied with a new approach, and in the present paper we   further pursue this novel point of view. Indeed,  we consider the family
\[
\Omega = \big( \mathcal L (G) \big)_G
\]
of systems of sets of lengths $\mathcal L (G)$, where $G$ is running through a set of representatives of non-isomorphic finite abelian groups.  If
$G'$ is a subgroup of a group $G$, then $\mathcal L (G') \subset \mathcal L (G)$ and, given two groups $G_1$ and $G_2$, we have
$\mathcal L (G_i) \subset \mathcal L (G_1 \oplus G_2)$ for each $i \in [1,2]$. Thus $\Omega$ is a directed family and, by Theorem A and \cite[Theorem 7.4.1]{Ge-HK06a}, we have
\[
\mathcal L (C_1) = \mathcal L (C_2) \subsetneq \mathcal L (C_3) = \mathcal L (C_2 \oplus C_2) \subsetneq \mathcal L (G) \subsetneq \mathcal L (G^*)
\]
for every finite abelian group $G$ with $\mathsf D (G) \ge 4$ and every infinite abelian group $G^*$.

If $G$ and $G'$ are  finite abelian groups with $\mathcal L (G) = \mathcal L (G')$, then their Davenport constants are equal (Proposition \ref{2.3}). Furthermore, for every positive integer $m \in \N$ there are only finitely many non-isomorphic finite abelian groups $G$ with Davenport constant $\mathsf D (G)=m$. We consider,  for every $m \in \N$, the finite family
\[
\Omega_m = \big( \mathcal L (G) \big)_{G \ \text{with} \  \mathsf D (G)=m}
\]
of all systems $\mathcal L (G)$, where $G$ is running through a set of representatives of non-isomorphic finite abelian groups $G$ having Davenport constant $\mathsf D (G)=m$. Thus, the Characterization Problem has an affirmative answer if and only if the systems in $\Omega_m$ are pairwise distinct for all $m \ge 4$. We say that an element $\mathcal L (G)$ of a subfamily $\Omega'$ of $\Omega$ is
\begin{itemize}
\item {\it maximal} in $\Omega'$ if $\mathcal L (G) \subset \mathcal L (G')$  implies that $G=G'$ for every element $\mathcal L (G')$ in $\Omega'$,

\item {\it minimal} in $\Omega'$ if $\mathcal L (G') \subset \mathcal L (G)$ implies that $G=G'$ for every element $\mathcal L (G')$ in $\Omega'$, and

\item {\it incomparable} in $\Omega'$ if it is maximal and minimal in $\Omega'$, that is $\mathcal L (G)$ is not comparable to any other element of $\Omega'$.
\end{itemize}

Now we can  formulate the main result of the present paper. By Theorem {\bf A}, it is sufficient to consider finite abelian groups $G$ with Davenport constant $\mathsf D (G) \ge 4$.

\begin{theorem} \label{1.1}~

\begin{enumerate}
\item For $m \in [4,6]$, $\mathcal L (C_m)$ is minimal in $\Omega_m$, $\mathcal L (C_2^{m-1})$ is maximal in $\Omega_m$, and $\mathcal L (C_m) \subsetneq \mathcal L (C_2^{m-1})$.

\item For every $m \ge 7$, $\mathcal L (C_m)$ is incomparable  in $\Omega_m$ and $\mathcal L (C_2^{m-1})$ is incomparable  in $\Omega_m$.

\item For every $m \ge 5$, $\mathcal L (C_2^{m-4}\oplus C_4)$ is maximal in $\Omega_m$.

\item For every $n \ge 2$,   $\mathcal L (C_2 \oplus C_{2n})$ is maximal in $\Omega_{2n+1}$. Moreover, $\mathcal L (C_2 \oplus C_{2n})$ is minimal among all $\mathcal L (G)$ in $\Omega_{2n+1}$ stemming from groups $G$ with $\mathsf D (G) = \mathsf D^* (G)$.
\end{enumerate}
\end{theorem}

\begin{corollary} \label{1.2}~
Let $G_1$ and $G_2$ be  non-isomorphic finite abelian groups with $\mathsf D (G_1)=\mathsf D (G_2)=m \in [4,7]$.
\begin{enumerate}
\item If $m \in [4,6]$ and $\mathcal L (G_1) \subset \mathcal L (G_2)$, then $G_1$ is cyclic of order $m$, $G_2$ is an elementary $2$-group of rank $m-1$, and $\mathcal L (G_1) \subsetneq \mathcal L (G_2)$.

\item If $m=7$, then $\mathcal L (G_1)$ and $\mathcal L (G_2)$ are incomparable.
\end{enumerate}
\end{corollary}

Statements $1$ and  $2$ of Theorem \ref{1.1} will be proved, resp., in Section \ref{3} and  \ref{4}; while Statements $3$ and $4$ of Theorem \ref{1.1} and Corollary \ref{1.2} will be proved in Section \ref{5}. The proofs are based on deep results on invariants controlling the structure of sets of lengths as well as on a careful case by case analysis when handling small groups.  In Section \ref{2} we gather some  background information on systems  of sets of lengths over finite abelian groups,  and we refer to the survey \cite{Sc16a} for more information.

\section{Background on systems of sets of lengths over finite abelian groups} \label{2}

We denote by $\N$ the set of positive integers and by $\N_0$ the set of nonnegative integers. For integers $a, b \in \Z$, $[a, b] = \{x \in \Z \colon a \le x \le b \}$ is the discrete interval between $a$ and $b$. Let $L, L' \subset \Z$ be subsets of the integers. Then $L+L' = \{a+b \colon a \in L, b \in L' \}$ denotes their sumset, $\Delta (L) \subset \N$ denotes the set of successive distances of elements from $L$, and $k \cdot L = \{ ka \colon a \in L \}$ is the dilation of $L$ by $k$. If $L \subset \N$, then $\rho (L) = \sup (L) / \min (L)$ is the elasticity of $L$, and for $L = \{0\}$ we set $\rho (L)=1$.
Let $d \in \N$,  $\ell,\, M \in \N_0$, and  $\{0,d\} \subset
\mathcal D \subset [0,d]$. The set $L \subset \Z$ is called

\begin{itemize}
\item an {\it arithmetical multiprogression} \ ({\rm AMP}) \ with \ {\it   difference} \ $d$, \ {\it period} \ $\mathcal D$ \ and \ {\it length}
      \ $\ell$, \ if $L$ is an interval of \ $\min L + \mathcal D + d \Z$ \ (this means that $L$ is finite nonempty and $L = (\min L + \mathcal
D + d \mathbb Z) \cap [\min L, \max L]$),
      \ and \ $\ell$ \ is maximal such that \ $\min L + \ell d \in L$.

\item an {\it almost arithmetical multiprogression} \ ({\rm AAMP}) \ with \ {\it difference} \ $d$, \ {\it period} \ $\mathcal D$,
      \ {\it length} \ $\ell$ and \ {\it bound} \ $M$, \ if
      \[
      L = y + (L' \cup L^* \cup L'') \, \subset \, y + \mathcal D + d \Z
      \]
      where \ $L^*$ \ is an \ {\rm AMP} \ with difference $d$ (whence $L^* \ne \emptyset$),
      period $\mathcal D$
      and length $\ell$ such that \ $\min L^* = 0$, \ $L' \subset [-M, -1]$, \ $L'' \subset \max L^* + [1,M]$, \ and \
      $y \in \Z$.
\end{itemize}
For a set $P$, we denote by $\mathcal F (P)$ the free abelian monoid with basis $P$. If $a = \prod_{p \in P} p^{\mathsf v_p (a)} \in \mathcal F (P)$, then $|a| = \sum_{p \in P} \mathsf v_p (a) \in \N_0$ is the length of $a$, and $\supp (a) = \{ p \in P \colon \mathsf v_p (a) > 0 \} \subset P$ is the support of $a$.

\medskip
\begin{itemize}
\item[] {\it Throughout this section, let $G$ be an additively written finite abelian group, say $G \cong C_{n_1} \oplus \ldots \oplus C_{n_r}$ with $1 < n_1 \mid \ldots \mid n_r$,  and let $G_0 \subset G$ be a subset.}
\end{itemize}
\medskip

In the above decomposition, $r = \mathsf r (G)$ is the rank of $G$ and $n_r = \exp (G)$ is the exponent of $G$. A tuple $(e_1, \ldots, e_s)$ of nonzero elements of $G$, with $s \in \N$,  is called a basis of $G$ if $G = \langle e_1 \rangle \oplus \ldots \oplus \langle e_s \rangle$. For a fixed basis $(e_1, \ldots, e_s)$ of $G$ we write $e_I = \sum_{i \in I} e_i$ for every subset $I \subset [1,s]$. In particular, we have $e_{\emptyset} = 0$.

\medskip
\noindent
{\bf The monoid of zero-sum sequences over $G_0$.}
The elements of the free abelian monoid $\mathcal F (G_0)$ with basis $G_0$ are called sequences over $G_0$. Let
\[
S = g_1 \cdot \ldots \cdot g_{\ell} = \prod_{g \in G_0} g^{\mathsf v_g (S)} \in \mathcal F (G_0)
\]
be a sequence over $G_0$. Then $|S|=\ell \in \N_0$ is the length of $S$ and $\sigma (S) = g_1 + \ldots + g_{\ell} \in G$ is the sum of $S$. We set $-S = (-g_1) \cdot \ldots \cdot (-g_{\ell})$. The sequence $S$ is said to be {\it zero-sum free} if $\sum_{i \in I} g_i \ne 0$ for all nonempty subsets $I \subset [1, \ell]$.
The monoid
\[
\mathcal B (G_0) = \{ S \in \mathcal F (G_0) \colon \sigma (S)=0 \} \subset \mathcal F (G_0)
\]
of zero-sum sequences over $G_0$ is a saturated submonoid of $\mathcal F (G_0)$ and hence a Krull monoid. The set of atoms of $\mathcal B (G_0)$ (in other words, the set of the minimal zero-sum sequences over $G_0)$ is denoted by $\mathcal A (G_0)$. The set $\mathcal A (G_0)$ is finite and
\[
\mathsf D (G_0) = \max \{ |U| \colon U \in \mathcal A (G_0) \} \in \N_0
\]
is the {\it Davenport constant} of $G_0$. We have the following lower and upper bounds,
\[
\mathsf D^* (G) := 1 + \sum_{i=1}^r (n_i-1) \le \mathsf D (G) \le |G| \,,
\]
where the left inequality is an equality for groups of rank $r \le 2$ and for $p$-groups (for recent progress on the Davenport constant we refer to \cite{Li20a, Gi18a}). Furthermore,  $\mathsf d (G) := \mathsf D (G) - 1$ is the maximal length of a zero-sum free sequence over $G$.

\medskip
\noindent
{\bf The arithmetic of $\mathcal B (G_0)$.} The free abelian monoid $\mathsf Z (G_0) = \mathcal F ( \mathcal A (G_0))$ is the factorization monoid of $\mathcal B (G_0)$. Let $\pi \colon \mathsf Z (G_0) \to \mathcal B (G_0)$ denote the canonical epimorphism. For every $B \in \mathcal  B (G_0)$, $\mathsf Z (B) = \pi^{-1} (B)$ is the {\it set of factorizations} of $B$ and $\mathsf L (B) = \{ |z| \colon z \in \mathsf Z (B) \}$ is the {\it set of lengths} of $B$. Note that $\mathsf L (1_{\mathcal B (G_0)}) = \{0\}$, and we have $\mathsf L (B) = \{1\}$ if and only if $B \in \mathcal A (G_0)$. Then
\[
\mathcal L (G_0) = \{ \mathsf L (B) \colon B \in \mathcal B (G_0) \}
\]
is the {\it system of sets of lengths} of $\mathcal B (G_0)$. The systems $\mathcal L (G)$ are of high relevance because of transfer results in factorization theory. Indeed, if $H$ is a transfer Krull monoid over $G$, then $\mathcal L (H) = \mathcal L (G)$. Transfer Krull monoids include  commutative Krull monoids and Krull domains but also  classes of non-commutative Dedekind domains. We do not discuss these connections  here but refer to the surveys \cite{Ge16c, Ge-Zh20a}.

We recall the concept of  the $g$-norm, which is a powerful tool for the study of sets of lengths of zero-sum sequences over cyclic groups.
Let $g \in G$ with \ $\ord(g) = n \ge 2$. For a
sequence $S = (n_1g) \cdot \ldots \cdot (n_{\ell}g) \in \mathcal F
(\langle g \rangle)$, \ where $\ell \in \N_0$ and $n_1, \ldots, n_{\ell}
\in [1, n]$, \ we define
\[
\| S \|_g = \frac{n_1+ \ldots + n_{\ell}}n  \,.
\]
Note that $\sigma (S) = 0$ implies that   $n_1 + \ldots + n_{\ell}
\equiv 0 \mod n$ whence $\| S \|_g \in \N_0$.  Thus, $\| \cdot
\|_g \colon \mathcal B(\langle g \rangle) \to \N_0$ is a
homomorphism, and $\| S \|_g =0$ if and only if $S=1$. If $S \in \mathcal A (G_0)$, then $\|S\|_g \in [1, n-1]$, and if $\|S\|_g = 1$, then $S \in \mathcal A (G_0)$.
Arguing as above we obtain that
\[
\frac{\|A\|_g}{n-1} \le \min \mathsf L (A) \le \max \mathsf L (A) \le \|A\|_g \,.
\]

Next we define the distance of factorizations and the catenary degree. Two factorizations $z, z' \in \mathsf Z (G_0)$ can be written, uniquely up to the order of the terms, in the form
\[
z = U_1 \cdot \ldots \cdot U_k V_1 \cdot \ldots \cdot V_{\ell} \ \text{and} \ z' = U_1 \cdot \ldots \cdot U_k W_1 \cdot \ldots \cdot W_{m}
\]
where all $U_r, V_s, W_t \in \mathcal A (G_0)$ and all $V_i \ne W_j$ for all $i \in [1, \ell]$ and all $j \in [1, m]$. Then $\mathsf d (z, z') = \max \{\ell, m\}$ is the distance between $z$ and $z'$. For an element $B \in \mathcal B (G_0)$, the catenary degree $\mathsf c (B)$ is the smallest $N \in \N_0$ such that for any two factorizations $z, z' \in \mathsf Z (B)$ there are factorizations $z=z_0, z_1, \ldots, z_k=z'$ of $B$ such that $\mathsf d (z_{i-1}, z_i) \le N$ for all $i \in [1,k]$. Then $\mathsf c (G_0) = \max \{ \mathsf c (B) \colon B \in \mathcal B (G_0) \}$ is the {\it catenary degree} of $G_0$.

We say that the monoid $\mathcal B (G_0)$ (resp. $G_0$) is {\it half-factorial} if and only if $|L|=1$ for all $L \in \mathcal L (G_0)$. We denote by
\begin{equation} \label{definition-delta}
\Delta (G_0) = \bigcup_{L \in \mathcal L (G_0)} \Delta (L) \ \subset \N
\end{equation}
the {\it set of distances} of $G_0$ and we set
\[
\daleth (G_0) =  \max \{ \min (L \setminus \{2\}) \mid 2 \in L \in \mathcal L (G_0) \}  \,.
\]
If $\Delta (G_0) \ne \emptyset$, then $\min \Delta (G_0) = \gcd \Delta (G_0)$ and, by \cite[Theorems 1.6.3 and 3.4.10]{Ge-HK06a}), we have
\begin{equation} \label{basic inequalities}
\daleth (G_0) \le 2 + \max \Delta (G_0) \le \mathsf c (G_0) \le \mathsf D (G_0) \,.
\end{equation}
The {\it set of minimal distances} $\Delta^* (G) \subset \Delta (G)$ is defined as
\[
\Delta^* (G) = \{ \min \Delta (G_0) \colon G_0 \subset G \ \text{with} \ \Delta (G_0) \ne \emptyset \} \subset \Delta (G) \,.
\]
For $k \in \N$, the $k$th {\it elasticity} of $G_0$ is defined as
\begin{equation} \label{definition-rho}
\rho_k (G_0) = \max \{ \max L \colon k \in L \in \mathcal L (G_0) \}
\ \text{and} \
\rho (G_0) = \sup \{ \rho (L) \colon L \in \mathcal L (G_0) \}
\end{equation}
is the {\it elasticity} of $G_0$.

\smallskip
We end this section with three propositions. They gather some of the key properties and results on the above invariants.
The first proposition reveals the relevance of $\Delta^* (G)$  (see \cite[Theorem 4.4.11]{Ge-HK06a}).

\begin{proposition} \label{2.1}
Let $G$ be a finite abelian group with $|G| \ge 3$. There exists some $M \in \N_0$ such that every set of lengths $L \in \mathcal L (G)$ is an {\rm AAMP} with difference $d \in \Delta^* (G)$ and bound $M$.
\end{proposition}

\begin{proposition} \label{2.2}
Let $m \ge 3$.
\begin{enumerate}
\item $\Delta (C_m) = \Delta (C_2^{m-1}) = [1, m-2]$ and $\Delta^* (C_m) \subset [1, m-2] = \Delta^* (C_2^{m-1})$.

\item $\max \Delta^* (G) = \max \{\exp (G)-2, \mathsf r (G)-1 \}$.

\item Let $k \in \N$. Then $\rho_{2k} (G) = k \mathsf D (G)$, $k \mathsf D (G) + 1 \le \rho_{2k+1} (G) \le k \mathsf D (G) + \mathsf D (G)/2$, and $\rho (G)  =  \mathsf D (G)/2$. If $G$ is cyclic, then $\rho_{2k+1}(G) = k \mathsf D (G)+1$ and if $G$ is an elementary $2$-group, then $\rho_{2k+1} (G) = k \mathsf D (G) + \lfloor \mathsf D (G)/2 \rfloor$.
\end{enumerate}
\end{proposition}

\begin{proof}
The claim on $\max \Delta^* (G)$ follows from \cite{Ge-Zh16a}. If $G$ is cyclic, then $\rho_{2k+1}(G) = k \mathsf D (G)+1$ for all $k \in \N$ by \cite[page 75, Theorem 5.3.1]{Ge-Ru09}. Proofs of the remaining claims can be found in \cite[Chapter 6]{Ge-HK06a}.
\end{proof}

\begin{proposition} \label{2.3}
Let $G$ and $G'$ be finite abelian groups with $\mathsf D (G) \ge 4$ such that $\mathcal L (G) = \mathcal L (G')$.
\begin{enumerate}
\item $\rho (G) = \rho (G')$ and $\rho_k (G) = \rho_k (G')$ for all $k \in \N$. In particular, $\mathsf D (G) = \mathsf D (G')$.

\item $\Delta (G) = \Delta (G')$ and $\max \Delta^* (G) = \max \Delta^* (G')$.

\item If $G$ is  cyclic or an elementary $2$-group with $\mathsf D (G) \ge 4$, then $G \cong G'$.
\end{enumerate}
\end{proposition}

\begin{proof}
1. The claims on $\rho (G)$ and on $\rho_k (G)$ follow from Definition \eqref{definition-rho}. Since $\mathsf D (G) = \rho_2 (G)$ and $\mathsf D (G') = \rho_2 (G')$ by Proposition \ref{2.2}, we obtain that $\mathsf D (G) = \mathsf D (G')$.

2. The claim on $\Delta (G)$ is clear by Definition \eqref{definition-delta}. The claim on $\Delta^* (G)$ is based on Proposition \ref{2.1} and is given in  \cite[Corollary 4.3.16]{Ge-HK06a}.

3. This follows from \cite[Theorem 7.3.3]{Ge-HK06a}.
\end{proof}

\section{Proof of Theorem \ref{1.1}.1} \label{3}

The goal of this section is to prove Theorem \ref{1.1}.1. A main part is to show that $\mathcal L (C_6) \subset \mathcal L (C_2^5)$. This will be done in a series of subsections.
We need  a lot of computations with zero-sum sequences over  $C_2^5$. To simplify notation and to avoid repetitions, we fix the following notation until the end of this section.

We  fix a basis $(e_1, \ldots, e_5)$ of $C_2^5$. For every subset $I \subset [1,5]$, we define
\[
e_I = \sum_{i \in I} e_i , \quad U_I = e_I \prod_{i\in I}e_i, \quad \text{ and }  \quad V_I = e_I \prod_{i\in [0,5] \setminus I}e_i \,.
\]
Moreover, we set $e_0  := e_{[1,5]}$ and $U := U_{[1,5]}$.
If $\emptyset \neq I  \subsetneq [1,5]$, then $U_I$ and $V_I$ are minimal zero-sum sequences over
$\{e_I, e_0, \dots, e_5\}$.

\subsection{On intervals in $\mathcal L (C_6)$ and $\mathcal L (C_2^5)$} \label{3.a}

The goal of this subsection is to show that  all intervals, that lie  in $\mathcal L (C_6)$, also lie in $ \mathcal L (C_2^5)$.
We start with two lemmas.

\begin{lemma} \label{3.1}
Let $L \in \mathcal  L (C_6)$ with $\{2, 5\} \subset L$. Then $L = \{2, 5\}$ or $L = \{2, 4, 5\}$, and both sets actually lie in $\mathcal L (C_6)$.
In particular, $[2,5] \notin \mathcal  L (C_6)$.
\end{lemma}

\begin{proof}
Let $B \in \mathcal B (C_6)$ with $\{2,5\} \subset \mathsf L (B)$.
Then $B = U_1 U_2$ with  $U_1, U_2 \in \mathcal A (C_6)$ and $|U_i| \ge 5$ for $i \in [1,2]$.
If $g \in C_6$ with $\ord (g)=6$, then  $W=g^6$,  $V=g^4(2g)$, $-W$, and $-V$ are the  atoms of length at least $5$.
Since $\mathsf L ((-W)W)= \{2,6\}$, $\mathsf L ((-V)V)= \{2,4,5\}$ and $\mathsf L ((-W)V)= \mathsf L ((-V)W)=\{2,5\}$,
the claim follows.
\end{proof}

\begin{lemma} \label{3.2}~
\begin{enumerate}
\item $[2,3], [2,4], [3,6], [3,7], [4,9], [4,10], [4,11] \in \mathcal  L (C_2^5)$. Moreover,  $\{2,4\}$, $\{2,5\}$, $\{2,6\}$, $\{2,3,5\}$, and  $\{2,4,5\}$ are in $\mathcal  L (C_2^5)$.

\item For each $k \in \mathbb{N}_{\ge 2}$ we have  $[2k, 6k -4], [2k,6k-3],[2k,6k-2], [2k,6k-1] \in \mathcal  L (C_2^5)$.

\item For each $k \in \mathbb{N}$ we have  $[2k+1, 6k -2], [2k+1,6k-1],[2k+1,6k], [2k+1,6k+1] \in \mathcal  L (C_2^5)$.
\end{enumerate}
\end{lemma}

\begin{proof}
1. We have $\mathsf L ( U_{[1,3]}^2) = \{2,4\}$, $\mathsf L ( U_{[1,4]}^2) = \{2,5\}$, and $\mathsf L ( U_{[1,5]}^2) = \{2,6\}$. All remaining sets, apart from
  $[4,11]$,  are already in $\mathcal L (C_2^4)$ (\cite[Theorem 4.8 and Proposition 4.10]{Ge-Sc-Zh17b}).
It remains to verify that $[4,11] \in \mathcal  L (C_2^5)$. To do so we define
\begin{equation} \label{V1V2}
 V_1 = e_1e_2e_3e_4e_{\{3,4,5\}}e_{\{1,2,5\}}, \quad \text{and} \quad V_2 = e_1e_{\{1,2\}}e_3e_4e_5e_{[2,5]}
\end{equation}
and we assert that $\mathsf L (U^2V_1V_2) = [4,11]$. Since $A = U^2V_1V_2$ is not a square and $|A|=24$, it follows that $\max \mathsf L (A) \le 11$ whence $\mathsf L (A) \subset [4,11]$. We assert  that $\mathsf L (UV_1) = \{2,4,5\}$ and that $\mathsf L (UV_2) = \{2,3,5\}$.

Let $z$ be a factorization of $UV_1$. The atoms, that divide $z$ and
 contain $e_{\{1,2,5\}}$,  are
$U_{\{1,2,5\}}$, $V_{\{1,2,5\}}$, $V_1$, and $V_1' = e_0e_5e_{\{3,4,5\}}e_{\{1,2,5\}}$.
If $z$ is divisible by  $U_{\{1,2,5\}}$, then $z = U_{\{1,2,5\}}V_{\{3,4,5\}}e_3^2e_4^2$, whence $|z|=4$.
If $z$ is divisible by  $V_{\{1,2,5\}}$, then $z=V_{\{1,2,5\}}U_{\{3,4,5\}}e_1^2e_2^2$, whence $|z|=4$.
If $z$ is divisible by  $V_1$, then $z=UV_1$, whence $|z|=2$.
If $z$ is divisible by   $V_1'$, then $z=V_1' e_1^2e_2^2e_3^2e_4^2$, whence $|z|=5$. This shows
$\mathsf L (UV_1) = \{2,4,5\}$.

Let $z$ be a factorization of  $UV_2$. The atoms, that divide $z$ and  contain $e_{\{1,2\}}$, are
$U_{\{1,2\}}$, $V_{\{1,2\}}$, $V_2$, and $V_2' = e_0e_2e_{[2,5]}e_{\{1,2\}}$.
If $z$ is divisible by   $U_{\{1,2\}}$, then $z=U_{\{1,2\}}V_{[2,5]}e_3^2e_4^2e_5^2$, whence $|z|=5$.
If $z$ is divisible by   $V_{\{1,2\}}$, then $z=V_{\{1,2\}}U_{[2,5]}e_1^2$, whence $|z|=3$
If $z$ is divisible by   $V_2$, then  $z=UV_2$, whence $|z|=2$.
If $z$ is divisible by  $V_2'$, then  $z=V_2' e_1^2e_3^2e_4^2e_5^2$, whence $|z|=5$. This shows
$\mathsf L (UV_2) = \{2,3,5\}$.

Consequently,  $[4,10] = \{2,3,5\}+\{2,4,5\} \subset \mathsf L (U^2 V_1V_2)$. Since
\[
V_1V_2 = \Big( e_1^2 \Big) \Big( e_3^2 \Big) \Big( e_4^2 \Big) \Big( e_{\{1,2,5\}}e_{\{1,2\}}e_5 \Big) \Big( e_{[2,5]}e_2e_{\{3,4,5\}} \Big) \,,
\]
and since $\mathsf L (U^2) = \{2,6\}$, it follows that $5 \in \mathsf L (V_1V_2)$, whence $11 \in \mathsf L (U^2V_1V_2)$.

\noindent
2. and 3.  (i) \underline{Claim 1:} $[2k+1,6k+1] \in \mathcal  L (C_2^5)$ for each $k \in \mathbb{N}$.

Let $k \in \N$. We recall from \cite[Proposition 4.10, proof of assertion A1]{Ge-Sc-Zh17b} that
\begin{equation} \label{U_1'}
U_1'  = e_{[1,4]} e_1\cdot \ldots \cdot e_4, \ U_2' = e_1e_2e_{\{1,3\}}e_{\{2,4\}}e_{\{3,4\}}, \quad \text{and} \quad U_3'  = e_{\{1,3\}}e_{\{2,4\}}e_{\{3,4\}}e_3e_4
\end{equation}
are atoms of lengths $5$ and that $\mathsf{L}(U_1'U_2'U_3') = [3,7]$.
We assert that  $\mathsf{L}(U^{2k-2}U_1'U_2'U_3') = [2k+1,6k+1]$.
Since $\mathsf{L}(U^{2k-2}) = 2k-2 + 4\cdot [0,k-1]$, it follows that
$2k-2 + 4\cdot [0,k-1] + [3,7] = [2k+1,6k+1] \subset \mathsf{L}(U^{2k-2}U_1'U_2'U_3')$. It remains to show the converse inclusion.
Since $|U^{2k-2}U_1'U_2'U_3'| = 6(2k-2) + 15$, it follows that
the minimal length is at least  $(12k+3)/6$ and as it is an integer it is at least $2k+1$. Moreover, as $0$ does not occur in this sequence, the maximal length is at most $(12k + 3)/2$ and thus it is at most $6k+1$.
This shows that $[2k+1,6k+1] \in \mathcal  L (C_2^5)$.

(ii) \underline{Claim 2:} $[2k, 6k -4] \in \mathcal  L (C_2^5)$ for each $k \ge 2$.

Since $1 + [2k + 1, 6k+1] = [2(k+1), 6(k+1)-4]$,  we conclude that    $[2k, 6k -4] \in \mathcal  L (C_2^5)$ for each $k \in \mathbb{N}_{\ge 2}$.

(iii) \underline{Claim 3:} $[2k,6k-1] \in \mathcal  L (C_2^5)$ for each $k \ge 2$.

Let $k \ge 2$. We consider $\mathsf L (U^{2k-2}V_1V_2)$ with $V_1, V_2$ as in \eqref{V1V2}.
Since $\mathsf L (U^{2}V_1V_2) = [4,11]$ and $\mathsf L (U^{2k-4})=2k-4 + 4 \cdot [0,k-2]$, it follows that
$2k-4 + 4 \cdot [0,k-2] + [4, 11] = [2k , 6k - 1] \subset \mathsf L (U^{2k-2}V_1V_2) $.
To prove the converse inclusion, it suffices to note that $|U^{2k-2}V_1V_2| = 12k$, and thus the minimal length it is at least $2k$.
Since the sequence does not contain $0$ and is not a square it follows that the maximal length is less than $6k$.

(iv) \underline{Claim 4:}   $[2k,6k-2] \in \mathcal  L (C_2^5)$ for each $k \ge 2$.

Let $k \ge 2$. We recall from \cite[Proposition 4.10, proof of assertion A3]{Ge-Sc-Zh17b} that $\mathsf{L}(U_1'^2 U_2'^2) = [4,10]$, where $U_1'$ and $U_2'$ are as defined above.
Now, we consider $\mathsf L (U^{2k-4}U_1'^2 U_2'^2)$, with $U_1'$ and $U_2'$ as above.
It follows that  $2k-4 + 4 \cdot [0,k-2] + [4, 10] = [2k , 6k - 2] \subset \mathsf L (U^{2k-4}U_1'^2 U_2'^2) $.
To prove the converse inclusion, it suffices to note that $|U^{2k-4}U_1'^2 U_2'^2| = 12k-4$, and thus the minimal length it is at least $(12k-4)/6$ and thus at least $2k$. Since the sequence does not contain $0$, it follows that that maximal length is at most $6k-2$ and the argument is complete.

(v) \underline{Claim 5:}   $[2k,6k-3] \in \mathcal  L (C_2^5)$  for each $k \ge 2$.

Let $k \ge 2$. We recall from \cite[Proposition 4.10, proof of assertion A2]{Ge-Sc-Zh17b} that $\mathsf{L}(U_1'^2 U_2'U_4') = [4,9]$, where $U_1'$ and $U_2'$ are as in \eqref{U_1'} and $U_4'=e_{\{1,2\}}e_{\{1,3\}}e_{\{2,4\}}e_{\{3,4\}}$.
We consider $\mathsf L (U^{2k-4}U_1'^2 U_2'U_4')$, and obtain    $2k-4 + 4 \cdot [0,k-2] + [4, 9] = [2k , 6k - 3] \subset \mathsf L (U^{2k-4}U_1'^2 U_2'U_4') $.
For converse inclusion, we note that $|U^{2k-4}U_1'^2 U_2'U_4'| = 12k-5$, and thus the minimal length  is  at least  $2k$ while the maximal length is at most $6k-3$.

(vi) \underline{Claim 6:}   $[2k+1, 6k -2], [2k+1,6k-1],[2k+1,6k] \in \mathcal  L (C_2^5)$ for each $k \in \N$.

Noting that  $[2k+1, 6k -2] =  1 + [2k, 6k -3]$, $ [2k+1,6k-1] = 1 + [2k, 6k-2]$, $[2k+1,6k] = 1 + [2k, 6k-1]$ and the latter intervals are  in $\mathcal L (C_2^5)$ for $k \ge 2$, it remains to study the case $k=1$.
By 1., we have $[3,6] \in \mathcal L (C_2^5)$. Moreover,  $[2,3], [2,4] \in \mathcal L (C_2^5)$, implies  $[3,4] = 1 + [2,3 \in \mathcal L (C_2^5)$ and $[3,5] = 1 + [2,4] \in \mathcal L (C_2^5)$.
\end{proof}

\begin{proposition} \label{3.3}
Every $L \in \mathcal{L}(C_6)$, that is  an interval, lies in  $\mathcal{L}(C_2^5)$.
\end{proposition}

\begin{proof}
Let $L \in \mathcal L (C_6)$ be an interval. The claim holds if $L$ is singleton. Suppose that $|L|\ge 2$.
We set $m = \min L$ and $n = \max L$. By Proposition \ref{2.2}, we have $\rho(C_6) = 3$,  $\rho_{m}(C_6) = 3m$ for even $m$, and  $\rho_{m}(C_6) = 3(m-1) + 1$ for odd $m$.
Thus $n \le 3m$, and we assert  that $n < 3m$. Assume to the contrary that $n = 3m$, and let $B \in \mathcal B (C_6)$ with $\mathsf L (B)=L$.
If $0 \mid B$, then $\min \mathsf{L}(0^{-1}B) = m-1$ and  $\max \mathsf{L}(0^{-1}B) = 3m-1$,
a contradiction to    $\rho(C_6) = 3$. Thus  $0 \nmid B$.  Moreover, each atom in a factorization of length $m$ must have length $6$.
However, the only two minimal zero-sum sequences of length $6$ over $C_6$ are $g^6$ and $(-g)^6$, where $g$ is a generating element of $C_6$. Thus $\supp(B) \subset \{-g, g\}$. However, $\Delta( \{-g, g\}) = \{4\}$, contradicting the assumption that $\mathsf{L}(B)=L$ is an interval.

We now write $n = m + l$ with $l \in \N$.
If $l\le 2$, then  $[m,m+l]=  (m-2) + [2, 2+l] \in \mathcal  L (C_2^5)$ by Lemma \ref{3.2}.
Suppose that  $l = 3$. Since $[2,5] \notin \mathcal L (C_6)$ by Lemma \ref{3.1}, it follows that  $m \ge 3$, whence $[m,m+l]=  (m-3) + [3, 6] \in \mathcal  L (C_2^5)$ by Lemma \ref{3.2}. Now we suppose  that $l \ge 4$ and distinguish two cases.

CASE 1: $m$ is even, say $m =2k'$.

By the argument above, we get $n < 3m = 6k'$.
Thus, $l \le 4k' - 1$, say $l = 4k - i$ with $k \in [2,k']$ and $i \in [1,4]$. Then $m-2k \ge 0$ and  $[m,m+l] = m-2k +[2k, 6k - i] \in \mathcal  L (C_2^5) $ by Lemma \ref{3.2}.

CASE 2: $m$ is odd, say $m =2k'+1$.

By the argument above, we get $n \le 6k'+1$. Thus $l \le 4k'$, say
$l = 4k - i $ with $k \in [1,k']$ and $i \in [0,3]$. Then $m-(2k+1) \ge 0$ and $[m,m+l] = m-(2k+1) +[2k+1, 6k+1 - i] \in \mathcal  L (C_2^5) $ by Lemma \ref{3.2}.
\end{proof}

\subsection{On AMPs with periods $\{0,1,4\}$ and $\{0,3,4\}$ in $\mathcal L (C_6)$ and $\mathcal L (C_2^5)$} \label{3.b}

The goal of this subsection is to show that all AMPs with period $\{0,1,4\}$ or with $\{0,3,4\}$, that lie in $\mathcal L (C_6)$, also lie in $\mathcal L (C_2^5)$.

\begin{lemma} \label{3.4}~
The sets $\{3,4,7\}, \ \{3,6,7\}, \ \{4,5,8,9\}, \ \{4,7,8,11\}$, and $\{5,8,9,12,13\}$ lie in $\mathcal L (C_6)$ and in  $\mathcal L (C_2^5)$.
\end{lemma}

\begin{proof}
1. First we show that the  given sets lie in $\mathcal L (C_6)$. Let $g \in G$ with $\ord (g)=6$ and set
\[
A = (2g)g^{v+4}(-g)^{w+2} \,, \quad \text{where $v , w \in \N_0$ with $v\equiv w +2 \mod{6}$.}
\]

CASE 1: $v \equiv 0 \mod{6}$.

By \cite[Lemma 3.6]{Ge-Sc19d}, we have
\[
\begin{aligned}
\mathsf L (A) & = 1 + (\mathsf L (  g^v (-g)^{w+2} )  \cup  \mathsf L (  g^{v+4} (-g)^{w})) , \\
\mathsf L ( g^v (-g)^{w+2}) & = \frac{v+w+2}{6} + 4 \cdot [0, \frac{\min\{v,w+2\}}{6}] , \quad \text{and} \\
\mathsf L ( g^{v+4} (-g)^{w}) &  = 3 + \frac{v+w+2}{6} + 4 \cdot [0, \frac{\min\{v,w-4\}}{6}] \,.
\end{aligned}
\]
If $v=6$ and $w=4$, then $\mathsf L (A) = \{3,6,7\}$.
If $v=6$ and $w=10$, then $\mathsf L (A) = \{4,7,8,11\}$. If $v=12$ and $w=10$, then $\mathsf L (A) = \{5,8,9,12,13\}$.

CASE 2:  $v \equiv 2 \mod{6}$.

By \cite[Lemma 3.6]{Ge-Sc19d}, we have
\[
\begin{aligned}
\mathsf L (A) & = 1 + (\mathsf L (  g^v (-g)^{w+2} )  \cup  \mathsf L (  g^{v+4} (-g)^{w})) , \\
\mathsf L ( g^v (-g)^{w+2})&  =  1 + \frac{v+4 +  w}{6} + 4 \cdot [0, \frac{\min\{v-2,w\}}{6}] , \quad \text{and} \\
\mathsf L ( g^{v+4} (-g)^{w}) & =  \frac{v+4 +w}{6} + 4 \cdot [0, \frac{\min\{v+4,w\}}{6}] \,.
\end{aligned}
\]
If $v=2$ and $w=6$, then $\mathsf L (A) = \{3,4,7\}$. If $v=8$ and $w=6$, then $\mathsf L (A) = \{4,5,8,9\}$.

2. Now we show that the sets lie in $\mathcal L (C_2^5)$. If $A_1 = U^2 ( e_1  e_2 e_{[1,2]})$, then
\[
\begin{aligned}
A_1 & = U e_1^2  e_2^2 V_{[1,2]}  = e_1^2 \cdot \ldots \cdot e_5^2  e_0^2 U_{[1,2]} \,,
\end{aligned}
\]
whence $\mathsf L (A_1) = \{3,4,7\}$. If $A_2 = U^2 ( e_1 \cdot \ldots \cdot e_4 e_{[1,4]})$, then
\[
\begin{aligned}
A_2 & = U e_1^2 \cdot \ldots \cdot e_4^2  V_{[1,4]}  = e_1^2 \cdot \ldots \cdot e_5^2  e_0^2 U_{[1,4]} \,,
\end{aligned}
\]
whence $\mathsf L (A_2) = \{3,6,7\}$. Note that $\mathsf L (U^2A_2) = \{3,6,7\}+\{2,6\} = \{5,8,9,12,13\}$.
 If $A_3 = U^3 U_{[1,2]}$, then
\[
\begin{aligned}
A_3 & = U^3 U_{[1,2]}  = U e_1^2 \cdot \ldots \cdot e_5^2 e_{0}^2 U_{[1,2]}  = U^2 e_1^2 e_2^2 V_{[1,2]}  = (e_1^2)^2 (e_2^2)^2 e_3^2 e_4^2 e_5^2 e_{0}^2 V_{[1,2]} \,,
\end{aligned}
\]
whence $\mathsf L (A_3) = \{4,5,8,9\}$. If $A_4=U^3 V_{[1,2]}$, then
\[
\begin{aligned}
A_4 & = U^3 V_{[1,2]}  = U^2 U_{[1,2]} e_3^2 e_4^2 e_5^2 e_{0}^2
   = U e_{0}^2e_1^2 e_2^2e_3^2e_4^2e_5^2  V_{[1,2]}  = U_{[1,2]} e_1^2 e_2^2 (e_3^2)^2 (e_4^2)^2 (e_5^2)^2 (e_0^2)^2 \,,
\end{aligned}
\]
whence $\mathsf L (A_4) = \{4,7,8,11\}$.
\end{proof}

\begin{proposition} \label{3.5}~
Every $L \in \mathcal L (C_6)$, that is an {\rm AMP} with period $\{0,3,4\}$, lies in $\mathcal L (C_2^5)$.
\end{proposition}

\begin{proof}
Let $L \in \mathcal L (C_6)$ be an AMP with period $\{0,3,4\}$ and set $m = \min L$. If $L$ is a singleton, then the claim holds. Suppose that $L$ is not a singleton. Then there is $k \in \N_0$ such that $L$ has one of the following two forms.
\begin{itemize}
\item $L = \{m,m+3\} + 4 \cdot [0,k]$.

\item $L = \big( \{m,m+3\} + 4 \cdot [0,k] \big) \cup \{m+4(k+1) \}$.
\end{itemize}
We distinguish two cases.

CASE 1: $L = \{m,m+3\} + 4 \cdot [0,k]$ for some $k \in \N_0$.

Then
\[
\rho (L) = \frac{m+4k+3}{m} \le \rho (C_6) = 3 \,,
\]
whence $m \ge 2k+2$. Clearly, if $\{2k+2, 2k+5\} + 4 \cdot [0,k] \in \mathcal L (C_2^5)$, then the same is true for $L$. Thus we may assume that $m = 2k+2$. If $k=0$, then $L=\{2,5\} \in \mathcal L (C_2^5)$ by Lemma \ref{3.2}. Now suppose that $k \ge 1$. Then we have
\[
\begin{aligned}
L = \{2k+2, 2k+5\} + 4 \cdot [0,k] & = \{2k+2, 2k+5, 2k+6, 2k+9\} + 4 \cdot [0,k-1] \\
 & =  \{4,7,8,11\} + 2(k-1)+ 4 \cdot [0,k-1] \,.
\end{aligned}
\]
If  $A_4 = U^3 V_{[1,2]}$, then $\mathsf L (A_4) = \{4,7,8,11\}$ by   Lemma \ref{3.4}, whence
\[
\mathsf L \big( A_4 U^{2(k-1)}  \big) = \mathsf L (C) + \mathsf L \big( U^{2(k-1)}  \big) = \{4,7,8,11\} + 2(k-1)+ 4 \cdot [0,k-1] = L \,.
\]

\smallskip
CASE 2: $L = \big( \{m,m+3\} + 4 \cdot [0,k] \big) \cup \{m+4(k+1)\}$ for some $k \in \N_0$.

Then
\[
\rho (L) = \frac{m+4k+4}{m} \le \rho (C_6) = 3 \,,
\]
whence $m \ge 2k+2$. Assume to the contrary that  $m = 2k+2$.  Then $\rho (L)=3$. Since every set $L_0 \in \mathcal L (C_6)$ with $\rho (L_0) = 3$ is an arithmetical progression with difference $4$, it follows that  $\rho (L)<3$ and this  implies  $m \ge 2k+3$. Clearly, if $\big( \{2k+3, 2k+6\} + 4 \cdot [0,k] \big) \cup \{6k+7\} \in \mathcal L (C_2^5)$, then the same is true for $L$. Thus we may assume that $m = 2k+3$. Then we have
\[
\begin{aligned}
L & = \big( \{2k+3, 2k+6\} + 4 \cdot [0,k] \big) \cup \{6k+7\} \\
  & =  \{3,6\} + 2k + 4 \cdot [0,k] \cup \{6k+7\}  =  \{3,6,7\} + 2k + 4 \cdot [0,k] \,.
\end{aligned}
\]
If  $A_2 = U^2 U_{[1,4]}$, then $\mathsf L (A_2) = \{3,6,7\}$ by Lemma \ref{3.4}, whence
\[
\mathsf L \big( A_2 U^{2k} \big) = \{3,6,7\} + 2k + 4 \cdot [0,k] = L \,. \qedhere
\]
\end{proof}

\begin{proposition} \label{3.6}~
Every $L \in \mathcal L (C_6)$, that is an {\rm AMP} with period $\{0,1,4\}$, lies in $\mathcal L (C_2^5)$.
\end{proposition}

\begin{proof}
Let $L \in \mathcal L (C_6)$ be an AMP with period $\{0,1,4\}$ and set $m = \min L$. If $L$ is a singleton, then the claim holds. Suppose that $L$ is not a singleton. Then there is $k \in \N_0$ such that $L$ has one of the following two forms.
\begin{itemize}
\item $L = \{m,m+1\} + 4 \cdot [0,k]$.

\item $L = \big( \{m,m+1\} + 4 \cdot [0,k] \big) \cup \{m+4(k+1) \}$.
\end{itemize}
We distinguish two cases.

CASE 1: $L = \{m,m+1\} + 4 \cdot [0,k]$ for some $k \in \N_0$.

Then
\[
\rho (L) = \frac{m+4k+1}{m} \le \rho (C_6) = 3 \,,
\]
whence $m \ge 2k+1$. If $m=2k+1$, then $L= \{2k+1, 2k+2\} + 4 \cdot [0,k]$ and thus
\[
\max L = 6k+2 \le \rho_{2k+1}(C_6) = 6k+1 \,,
\]
a contradiction. Thus $m \ge 2k+2$. Clearly, if $\{2k+2, 2k+3\} + 4 \cdot [0,k] \in \mathcal L (C_2^5)$,  then the same is true for $L$. Thus we may assume that $m = 2k+2$. If $k=0$, then $L = \{2,3\} \in \mathcal L (C_2^5)$ by Lemma \ref{3.2}. Suppose that $k \ge 1$. Then
\[
\begin{aligned}
L & = \{2k+2, 2k+3\} + 4 \cdot [0,k] \\
  & = \{2k+2, 2k+3, 2k+6, 2k+7\} + 4 \cdot [0, k-1] \\
  & = \{2(k-1)+4, 2(k-1)+5, 2(k-1)+8, 2(k-1)+9\} + 4 \cdot [0,k-1] \\
  & = \{4,5,8,9\} + 2(k-1) + 4 \cdot [0, k-1] \,.
\end{aligned}
\]
If  $A_3 = U^3 U_{[1,2]}$, then $\mathsf L (A_3) = \{4,5,8,9\}$ by Lemma \ref{3.4}, whence
\[
\mathsf L \big( A_3 U^{2k-2} \big) = \{3,4,7\} + 2(k-1) + 4 \cdot [0,k-1] = L \,.
\]

\smallskip
CASE 2: $L = \big( \{m,m+1\} + 4 \cdot [0,k] \big) \cup \{m+4(k+1) \}$ for some $k \in \N_0$.

Then
\[
\rho (L) = \frac{m+4k+4}{m} \le \rho (C_6) = 3 \,,
\]
whence $m \ge 2k+2$. Assume to the contrary that  $m = 2k+2$.  Then $\rho (L)=3$. Since every set $L_0 \in \mathcal L (C_6)$ with $\rho (L_0) = 3$ is an arithmetical progression with difference $4$, it follows that  $\rho (L)<3$ and this  implies  $m \ge 2k+3$. Clearly, if $\big( \{2k+3, 2k+4\} + 4 \cdot [0,k] \big) \cup \{6k+7\} \in \mathcal L (C_2^5)$, then the same is true for $L$. Thus we may assume that $m = 2k+3$. Then we have
\[
\begin{aligned}
L & = \big( \{2k+3, 2k+4\} + 4 \cdot [0,k] \big) \cup \{6k+7\} \\
  & =  \{3,4\} + 2k + 4 \cdot [0,k] \cup \{6k+7\}  =  \{3,4,7\} + 2k + 4 \cdot [0,k] \,.
\end{aligned}
\]
If  $A_1 = U^2 U_{[1,2]}$, then $\mathsf L (A_1) = \{3,4,7\}$ by Lemma \ref{3.4}, whence
\[
\mathsf L \big( A_1 U^{2k} \big) = \{3,4,7\} + 2k + 4 \cdot [0,k] = L \,. \qedhere
\]
\end{proof}

\subsection{On AMPs with periods $\{0,1,2,4\}$, $\{0,1,3,4\}$, and $\{0,2, 3,4\}$ in $\mathcal L (C_6)$ and $\mathcal L (C_2^5)$} \label{3.c}

The goal of this subsection is to show that all AMPs with period $\{0,1,2,4\}$,  $\{0,1,3,4\}$, or $\{0,2,3,4\}$, that lie in $\mathcal L (C_6)$, also lie in $\mathcal L (C_2^5)$.
We start by determining all AMPs with these periods in $\mathcal L (C_6)$. To this end we make use of the arguments in \cite{Ge-Sc19d}, which already contains a fairly precise description of theses sets but stops short of giving a full characterization.

\begin{proposition} \label{AMP4_C6}
Let $L \in \mathcal{L}(C_6)$ with $|L|\ge 4$.
\begin{enumerate}
\item If $L$ is an {\rm AMP} with period $\{0,1,2,4\}$, then $L$  equals one of the following sets for some $y, k \in \mathbb{N}_0${\rm :}
\begin{enumerate}
\item $y + 2k + \{4,5,6,8\} + 4 \cdot [0,k]$.
\item $y + 2k + \{4,5,6,8,9\} + 4 \cdot [0,k]$.
\item $y + 2k + \{5,6,7,9,10,11\} + 4 \cdot [0,k]$.
\end{enumerate}
\item If $L$ is an {\rm AMP} with period $\{0,1,3,4\}$, then $L$  equals one of the following sets for some $y, k \in \mathbb{N}_0${\rm :}
\begin{enumerate}
\item $y + 2k + \{3,4,6,7\} + 4 \cdot [0,k]$.
\item $y + 2k + \{4,5,7,8,9\} + 4 \cdot [0,k]$.
\item $y + 2k + \{5,6,8,9,10,12\} + 4 \cdot [0,k]$.
\end{enumerate}
\item If $L$ is an {\rm AMP} with period $\{0,2,3,4\}$, then $L$  equals one of the following sets for some $y, k \in \mathbb{N}_0${\rm :}
\begin{enumerate}
\item $y + 2k + \{3,5,6,7\} + 4 \cdot [0,k]$.
\item $y + 2k + \{4,6,7,8,10\} + 4 \cdot [0,k]$.
\item $y + 2k + \{4,6,7,8,10,11\} + 4 \cdot [0,k]$.
\end{enumerate}
\end{enumerate}
Moreover, all these sets actually are elements of $\mathcal{L}(C_6)$.
\end{proposition}

\begin{proof}
Clearly, it suffices to prove the claim for $y=0$.
The results from \cite{Ge-Sc19d}, specifically the proof of  Proposition 3.8 as well as the statements of Lemmas 3.3. to 3.7 show that
the only cases to consider are those that are treated in Lemmas 3.6 and 3.7 in that paper.  More specifically, AMPs with these periods arise in Case 2 of the proof of Lemma 3.6 and at the end of the proof of Lemma 3.7.

We start by considering the details of the sets that occur in Lemma 3.6 from \cite{Ge-Sc19d}.
As established there these sets arise as the sets of lengths of zero-sum sequences of the form $(2g)^2 g^{v'}(-g)^{w'}$ with $v,w\in \mathbb{N}_0$ where $g$ is a generating element of the group.

Specifically, the first sets of cardinality at least $4$ that arise in Case 2, the sets in Case 1 are not of the relevant form, are the set of lengths of $(2g)^2 g^{6}(-g)^{-m+4}$ with $m\in \mathbb{N}$, which are $\{3+m, 4+m, 5+m, 7+m\}$. These are thus exactly the sets of the form given in 1.a with $k=0$ and $y \in \mathbb{N}_0$ (note that $m \ge 1$).

Then, the sets of lengths of $A = (2g)^2 g^{v+8}(-g)^{w+4}$ with $v,w\in \mathbb{N}_0$ are considered.
More precisely one has (we refer to the argument there for missing details):

In case $v \equiv 0 \mod 6$, the set is an AMP with period $\{0,2,3,4\}$;
its minimum is $2 + (v+w + 4)/6$  and its maximum is the maximum of
$2 + (v+w + 4)/6  + 4 \min\{v+6, w+4\}/6$ and $2 + (v+w + 4)/6  + 4 \min\{v+6, w-2\}/6$
With $v=0 + 6k$ and $w = 2 + 2k$ for $k \in \mathbb{N}_0$, this yields exactly the sets in 3.a  with $y=0$ while with $v=0 + 6k$ and $w = 8 + 2k$ this yields exactly the sets in 3.b with $y=0$.

In case $v \equiv 2 \mod 6$, the set is an AMP with period $\{0,1,3,4\}$;
its minimum is $2 + (v+w + 6)/6$  and its maximum is the maximum of
$3 + (v+w + 6)/6  + 4 \min\{v+4, w+2\}/6$ and $5 + (v+w + 6)/6  + 4 \min\{v+4, w-4\}/6$.
With $v=2 + 6k$ and $w = 4 + 2k$ for $k \in \mathbb{N}_0$, this yields exactly the sets in 2.b with $y=0$ while with $v=2 + 6k$ and $w = 10 + 2k$ this yields exaxtly the sets in 2.c with $y=0$.

In case $v \equiv 4 \mod 6$, the set is an AMP with period $\{0,1,2,4\}$;
its minimum is $2 + (v+w + 8)/6$  and its maximum is the maximum of
$2 + (v+w + 8)/6  + 4 \min\{v+8, w\}/6$ and $4 + (v+w + 8)/6  + 4 \min\{v+2, w\}/6$.
For $k\ge 1$ with $v=4 + 6(k-1)$ and $w = 6+ 6k$, this yields exactly the sets in 1.a with $k \ge 1$ and $y=0$ while with $v=4 + 6k$ and $w = 6+ 6k$ for $k \in \mathbb{N}_0$ this yields exactly the sets in 1.c with $y=0$.

It remains to consider the sets arising in Lemma 3.7 from \cite{Ge-Sc19d}; it turns out these yield the sets in 1.b, 2.a and 3.b.
We see that the AMPs with these periods arise towards the end in Case 2.
Specifically, they arise as the sets of length of $A=(2g)(4g) g^{v}(-g)^{w}$ with $v,w\in \mathbb{N}_0$ where $g$ is a generating element of the group.
More concretely, $v$ and $w$ are congruent modulo $6$ and we set $v = r + 6m$ and $w = r + 6n$ with $r \in [0,5]$ and $m,n\in \mathbb{N}_0$.

If $r \in [4,5]$,  the set is an AMP with period $\{0,2,3,4\}$;
its minimum is $r-2 + m+n$  and its maximum is $r+1+m+n + 4 \min\{m,n\}$.
With $m=n=k+1$  for $k \in \mathbb{N}_0$, this yields exactly the sets in 3.c  with $y=0$ and $y=1$ for $r=4$ and $r=5$, respectively.

If $r \in [2,3]$,  the set is an AMP with period $\{0,1,2,4\}$;
its minimum is $r + m+n$  and its maximum is $r+1+m+n + 4 \min\{m,n\}$.
With $m=n=k+1$  for $k \in \mathbb{N}_0$, this yields exactly the sets in 1.b  with $y=0$ and $y=1$ for $r=2$ and $r=3$, respectively.

If $r \in [0,1]$,  the set is an AMP with period $\{0,1,3,4\}$;
its minimum is $r + 1 + m+n$  and its maximum is $r+1+m+n + 4 \min\{m,n\}$.
With $m=n=k+1$  for $k \in \mathbb{N}_0$, this yields exactly the sets in 2.a  with $y=0$ and $y=1$ for $r=0$ and $r=1$, respectively.
\end{proof}

In a series of lemmas  we  show that all  sets, listed in Proposition \ref{AMP4_C6}, lie in $\mathcal{L}(C_2^5)$.

\begin{lemma} \label{realC25_1}
Let $y,k\in \mathbb{N}_0$.
\begin{enumerate}
\item $y + 2k + \{3,4,6,7\} + 4 \cdot [0,k] \in \mathcal{L}(C_2^5)$.

\item $y + 2k + \{4,5,6,8,9\} + 4 \cdot [0,k] \in \mathcal{L}(C_2^5)$.
\end{enumerate}
\end{lemma}

\begin{proof}
It suffices to show the claim for $y=0$. Let $k \in \N_0$.

1. We set  $A = e_{[1,2]}^2 U^{2k+2}$ and assert that $\mathsf{L}(A)= 2k + \{3,4,6,7\} + 4 \cdot [0,k]$.
Every factorization $z$ of $A$ can be written as  $z_1z_2$, where $z_1$ is a factorization of $A_1=e_{[1,2]}^2 U^2$ and $z_2$ is a factorization of $U^{2k}$. Since $\mathsf{L}(U^{2k})= 2k + 4 \cdot [0,k]$, it suffices to show that $\mathsf{L}(e_{[1,2]}^2 U^{2}) = \{3,4,6,7\}$.

Let $z$ be a factorization of $e_{[1,2]}^2 U^2$. We do a case analysis depending on the atom dividing $z$ and containing the element $e_{[1,2]}$.
If $z$  is divisible by $e_{[1,2]}^2$, then  $|z|=3$ or $|z|=7$.
If $z$ is divisible by $U_{[1,2]}V_{[1,2]}$, then  $z=U_{[1,2]}V_{[1,2]}U$, whence $|z|=3$.  If $z$ is divisible by $U_{[1,2]}^2$, then $z=U_{[1,2]}^2 e_0^2 e_3^2e_4^2 e_5^2$, whence $|z|=6$.  If $z$ is divisible by $V_{[1,2]}^2$, then $z=V_{[1,2]}^2 e_1^2 e_2^2$, whence $|z|=4$.

2. We set  $A = U_{[1,2]}^2 U^{2k+2}$ and assert that  $\mathsf{L}(A)= 2k + \{4,5,6,8,9\} + 4 \cdot [0,k]$.
Every factorization $z$ of $A$ can be written as  $z_1z_2$, where $z_1$ is a factorization of $A_1=U_{[1,2]}^2 U^2$ and $z_2$ is a factorization of $U^{2k}$. Thus, it suffices to show that $\mathsf{L}(U_{[1,2]}^2 U^{2}) = \{4,5,6,8,9\} $.

Let $z$ be a factorization of $ U_{[1,2]}^2 U^{2}$. We do a case analysis depending on the atom dividing $z$ and containing the element $e_{[1,2]}$.
If $z$ is divisible by $U_{[1,2]}^2$, then $|z|=4$ or $|z|=8$.
If $z$ is divisible by $e_{[1,2]}^2$, then $z=e_{[1,2]}^2e_1^2e_2^2y$, where $y$ is a factorization of $U^2$, whence  $|z|=5$ or $|z|=9$.
If $z$ is divisible by $U_{[1,2]}V_{[1,2]}$, then $z=U_{[1,2]}V_{[1,2]}e_1^2e_2^2U$, whence $|z|=5$.  If $z$ is divisible by $V_{[1,2]}^2$, then $z=V_{[1,2]}^2 (e_1^2)^2 (e_2^2)^2$, whence $|z|=6$.
\end{proof}

\begin{lemma} \label{realC25_2}
Let $y,k\in \mathbb{N}_0$.
\begin{enumerate}
\item $y + 2k + \{3,5,6,7\} + 4 \cdot [0,k] \in \mathcal{L}(C_2^5)$.
\item $y + 2k + \{4,5,6,8\} + 4 \cdot [0,k] \in \mathcal{L}(C_2^5)$.
\item $y + 2k + \{5,6,7,9,10,11\} + 4 \cdot [0,k] \in \mathcal{L}(C_2^5)$.
\item $y + 2k + \{4,5,7,8,9\} + 4 \cdot [0,k] \in \mathcal{L}(C_2^5)$.
\item $y + 2k + \{5,6,8,9,10,12\} + 4 \cdot [0,k] \in \mathcal{L}(C_2^5)$.
\item $y + 2k + \{4,6,7,8,10\} + 4 \cdot [0,k] \in \mathcal{L}(C_2^5)$.
\end{enumerate}
\end{lemma}

\begin{proof}
It suffices to show the claim for $y=0$. Let $k \in \N_0$.
We set  $G_0 = \{e_0,e_1, e_2, e_3,e_4,e_5\} \cup \{e_{[1,2]}, e_{[3,4]}\}$ and observe that the atoms of $\mathcal B (G_0)$ of
 length at least three are $U$, $U_{[1,2]}$, $U_{[3,4]}$, $V_{[1,2]}$, $V_{[3,4]}$, and $W  = e_0e_{[1,2]}e_{[3,4]}e_5$.

1. We set  $A = W U^{2k+2}$ and assert that  $\mathsf{L}(A)= 2k + \{3,5,6,7\} + 4 \cdot [0,k]$.
Every factorization $z$ of $A$ can be written as  $z_1z_2$, where $z_1$ is a factorization of $A_1=W U^2$ and $z_2$ is a factorization of $U^{2k}$. Since $\mathsf{L}(U^{2k})= 2k + 4 \cdot [0,k]$, it suffices to show that $\mathsf{L}(W U^{2}) = \{3,5,6,7\}$.

Let $z$ be a factorization of $W U^{2}$.
We do a case analysis depending on the atoms dividing $z$ and  containing the elements $e_{[1,2]}$ and $e_{[3,4]}$.
If $z$ is divisible by $W$, then  $|z|=3$ or $|z|=7$.
If $z$  is divisible by $U_{[1,2]}U_{[3,4]}$, then $z=U_{[1,2]}U_{[3,4]}e_0^2 e_5^2U$, whence $|z|=5$.  If $z$ is divisible by $U_{[1,2]}V_{[3,4]}$, then $z=U_{[1,2]}V_{[3,4]} e_0^2 e_1^2e_2^2 e_5^2$, whence $|z|=6$. Similarly, if $z$ is divisible by $V_{[1,2]}U_{[3,4]}$, then $|z|=6$.

2. We set  $A = U_{[1,2]} U_{[3,4]}U^{2k+2}$ and assert that  $\mathsf{L}(A)= 2k + \{4,5,6,8\} + 4 \cdot [0,k]$.
Every factorization $z$ of $A$ can be written as  $z_1z_2$, where $z_1$ is a factorization of $A_1= U_{[1,2]} U_{[3,4]}U^{2}$ and $z_2$ is a factorization of $U^{2k}$. Thus, it suffices to show that $\mathsf{L}( U_{[1,2]} U_{[3,4]} U^{2}) = \{4,5,6,8\}$.

Let $z$ be a factorization of $U_{[1,2]} U_{[3,4]} U^{2}$.
We do a case analysis depending on the atoms dividing $z$ and containing the elements $e_{[1,2]}$ and $e_{[3,4]}$.
If $z$  is divisible by $ U_{[1,2]} U_{[3,4]}$, then $|z|=4$ or $|z|=8$.
If $z$ is divisible by $U_{[1,2]}V_{[3,4]}$, then $z=U_{[1,2]}V_{[3,4]} e_3^2 e_4^2U$, whence $|z|=5$. Similarly, if $z$ is divisible by $V_{[1,2]}U_{[3,4]}$, then $|z|=5$. If $z$ is divisible by  $V_{[1,2]} V_{[3,4]}$, then $z=V_{[1,2]} V_{[3,4]} e_1^2e_2^2e_3^2e_4^2$, whence $|z|=6$. If $z$ is divisible by $W$, then  $z=WU e_1^2e_2^2e_3^2e_4^2$, whence $|z|=6$.

3. We set  $A = U_{[1,2]} U_{[3,4]}U^{2k+3}$ and assert that $\mathsf{L}(A)= 2k + \{5,6,7,9, 10,11\} + 4 \cdot [0,k]$.
Every factorization $z$ of $A$ can be written as  $z_1z_2$, where $z_1$ is a factorization of $A_1= U_{[1,2]} U_{[3,4]}U^{3}$ and $z_2$ is a factorization of $U^{2k}$. Thus, it suffices to show that $\mathsf{L}( U_{[1,2]} U_{[3,4]} U^{3}) = \{5,6,7,9, 10,11\}$.

Let $z$ be a factorization of $U_{[1,2]} U_{[3,4]} U^{3}$. If $z$ is divisible by $U$, then 2. implies that  $|z| \in 1 +\{4,5,6,8\}  = \{5,6,7,9\}$.
Suppose that $z$ is  not divisible by $U$.
Then $z$ is divisible by $e_0^2$, because all atoms containing $e_0$, other than $e_0^2$ and $U$, contain $e_{[1,2]}$ or $e_{[3,4]}$.
Moreover, $z$  is divisible by exactly one of the atoms $W$, $V_{[1,2]}$, $V_{[3,4]}$.
If  $z$ is divisible by  $W$, then all other atoms dividing $z$ have length $2$, whence   $|z|=11$.
If  $z$ is divisible by  $V_{[1,2]}$, then $z$ is divisible by  $ U_{[3,4]}$,  and all other atoms dividing $z$ have length $2$, whence $|z|=10$.
If  $z$ is divisible by  $V_{[3,4]}$, then the same argument shows that $|z|=10$.

4.  We set  $A = U_{[1,2]} V_{[3,4]}U^{2k+2}$ and assert that $\mathsf{L}(A)= 2k + \{4,5,7,8,9\} + 4 \cdot [0,k]$.
Every factorization $z$ of $A$ can be written as  $z_1z_2$, where $z_1$ is a factorization of $A_1= U_{[1,2]} V_{[3,4]}U^{2}$ and $z_2$ is a factorization of $U^{2k}$. Thus, it suffices to show that $\mathsf{L}( U_{[1,2]} V_{[3,4]} U^{2}) = \{4,5,7,8,9\}$.

Let $z$ be a factorization of $U_{[1,2]} V_{[3,4]} U^{2}$.
We do a case analysis depending on the atoms dividing $z$ and containing the elements $e_{[1,2]}$ and $e_{[3,4]}$.
If $z$ is divisible by $ U_{[1,2]} V_{[3,4]}$, then  $|z|=4$ or $|z|=8$.
If $z$ is divisible by  $W$, then   $z=W (e_1^2)(e_2^2)z'$, where $z'$ is a factorization of $U^2$, whence $|z|=5$ or $|z|=9$.
If $z$ is divisible by $U_{[1,2]}U_{[3,4]}$, then $z= U_{[1,2]}U_{[3,4]} e_0^2 e_1^2 e_2^2 e_5^2 U$, whence $|z|=7$.
If $z$ is divisible by $ V_{[1,2]} V_{[3,4]}$, then $z=V_{[1,2]} V_{[3,4]} e_1^2 e_2^2 U $, whence $|z|=5$.
If $z$ is divisible by $ V_{[1,2]} U_{[3,4]}$, then $z=V_{[1,2]} U_{[3,4]} (e_1^2)^2(e_2^2)^2 e_0^2 e_5^2$, whence $|z|=8$.

5. We set  $A = U_{[1,2]} V_{[3,4]}U^{2k+3}$ and assert that  $\mathsf{L}(A)= 2k + \{5,6,8,9, 10, 12\} + 4 \cdot [0,k]$.
Every factorization $z$ of $A$ can be written as  $z_1z_2$, where $z_1$ is a factorization of $A_1= U_{[1,2]} V_{[3,4]}U^{3}$ and $z_2$ is a factorization of $U^{2k}$. Thus, it suffices to show that $\mathsf{L}( U_{[1,2]} V_{[3,4]} U^{3}) = \{5,6,8,9, 10, 12\}$.

Let $z$ be a factorization of $U_{[1,2]} V_{[3,4]} U^{3}$.
If $z$ is divisible by  $U$, then 4. implies that  $|z| \in 1 +\{4,5,7,8,9\}  = \{5,6,8,9,10\}$.
Suppose that $z$ is not divisible by  $U$.
Then $z$ is divisible by  $e_0^2$, because all atoms containing $e_0$, other than $e_0^2$ and $U$, contain $e_{[1,2]}$ or $e_{[3,4]}$. Since $\mathsf v_{e_0} (A_1) = 4$, it follows that $z$ is either  divisible by $V_{[1,2]}$ and $V_{[3,4]}$  or divisible by $(e_0^2)^2$.
In the former case, all atoms dividing $z$, other than $V_{[1,2]}$ and $V_{[3,4]}$, have length $2$, whence $|z|= 2 + (26 - 10)/2=10$. In the latter case, $z$ is divisible by $U_{[1,2]}$ and $U_{[3,4]}$, and all atoms dividing $z$, other than $U_{[1,2]}$ and $U_{[3,4]}$, have length $2$, whence $|z|= 2 + (26 - 6)/2=12$.

6.  We set  $A = W U^{2k+3}$ and assert that $\mathsf{L}(A)= 2k + \{4,6,7,8,10\} + 4 \cdot [0,k]$.
Every factorization $z$ of $A$ can be written as  $z_1z_2$, where $z_1$ is a factorization of $A_1=W U^3$ and $z_2$ is a factorization of $U^{2k}$. Thus,  it suffices to show that $\mathsf{L}(W U^{3}) =\{4,6,7,8,10\}$.

Let $z$ be a factorization of $WU^3$.
If $z$ is divisible by  $U$, then 1. implies that $|z| \in 1 +\{3,5,6,7\}  = \{4,6,7,8\}$.
Suppose that $z$ is not divisible by $U$. Then $z$ is divisible by  $e_0^2$, because all atoms containing $e_0$, other than $e_0^2$ and $U$, contain $e_{[1,2]}$ or $e_{[3,4]}$.
Since  $\mathsf v_{e_0} (A_1) =4$, it follows that $z$ is either divisible by  $V_{[1,2]}$ and $V_{[3,4]}$  or divisible by $(e_0^2)^2$.
In the former case, we obtain that  $z = V_{[1,2]}V_{[3,4]}e_0^2 \cdot \ldots \cdot e_5^2$, whence $|z|=8$.
In the latter case, $z$ is divisible by  $U_{[1,2]}$ and $U_{[3,4]}$ and all other atoms dividing $z$ have length $2$, whence  $|z|=2 + (22-6)/2 = 10$.
\end{proof}

\begin{lemma} \label{realC25_3}
For all $y,k\in \mathbb{N}_0$, we have $y + 2k + \{4,6,7,8,10,11\} + 4 \cdot [0,k] \in \mathcal{L}(C_2^5)$.
\end{lemma}

\begin{proof}
It suffices to prove the claim for $y=0$. Let $k \in \N_0$.
We  set $G_0 = \{e_0,e_1, e_2, e_3,e_4,e_5\} \cup \{e_{\{1,2,5\}}, e_{\{3,4,5\}}\}$ and observe that the atoms of $\mathcal B (G_0)$ of
 length at least three are $U$, $U_{\{1,2,5\}}$, $U_{\{3,4,5\}}$, $V_{\{1,2,5\}}$, $V_{\{3,4,5\}}$,
$W = e_1e_2e_3e_4 e_{\{1,2,5\}}e_{\{3,4,5\}}$, and $W' = e_0e_5e_{\{1,2,5\}}e_{\{3,4,5\}}$.

We set  $A = W U^{2k+3}$ and assert that $\mathsf{L}(A)= 2k + \{4,6,7,8,10,11\} + 4 \cdot [0,k]$.
Every factorization $z$ of $A$ can be written as  $z_1z_2$, where $z_1$ is a factorization of $A_1=W U^3$ and $z_2$ is a factorization of $U^{2k}$. Since $\mathsf{L}(U^{2k})= 2k + 4 \cdot [0,k]$, it suffices to show that $\mathsf{L}(A_1) = \{4,6,7,8,10,11\}$.

Let $z$ be a factorization of $WU^3$.  We do a case analysis depending on the atoms dividing $z$ and containing the elements  $e_{\{1,2,5\}}$ and $e_{\{3,4,5\}}$.
If $z$  is divisible by $W$, then   $|z|=4$ or $|z|=8$.
If $z$ is divisible by $W'$, then $z=W'e_1^2e_2^2e_3^2e_4^2 y$, where $y$ is a factorization of $U_2$, whence $|z|=7$ or $|z|=11$.
If $z$  is divisible by $U_{\{1,2,5\}}U_{\{3,4,5\}}$, then $z=U_{\{1,2,5\}}U_{\{3,4,5\}}e_0^2 e_1^2 e_2^2 e_3^2 e_4^2 U$, whence $|z|=8$.  If $z$ is divisible by $V_{\{1,2,5\}}V_{\{3,4,5\}}$, then $z=V_{\{1,2,5\}}V_{\{3,4,5\}} e_1^2 e_2^2 e_3^2 e_4^2 e_5^2U$, whence $|z|=8$.
If $z$ is divisible by $U_{\{1,2,5\}}V_{\{3,4,5\}}$, then $z=U_{\{1,2,5\}}V_{\{3,4,5\}} e_3^2 e_3^2y$,
where $y$ is a factorization of $U_2$, whence  $|z|=6$ or $|z|=10$. Similarly, if $z$ is divisible by $V_{\{1,2,5\}}U_{\{3,4,5\}}$, then $|z|=6$ or $|z|=10$.
Thus $\mathsf{L}(A_1) = \{4,6,7,8,10,11\}$.
\end{proof}

\begin{proposition} \label{3.11}
Every $L \in \mathcal{L}(C_6)$, that is an {\rm AMP} with period $\{0,1,2,4\}$, $\{0,1,3,4\}$, or with $\{0,2,3,4\}$, lies in $\mathcal{L}(C_2^5)$.
\end{proposition}

\begin{proof}
Let $L \in \mathcal{L}(C_6)$ be an {\rm AMP} with period $\{0,1,2,4\}$, or with $\{0,1,3,4\}$, or with $\{0,2,3,4\}$.
If $L$ is a singleton, then the claim is holds. If $|L|=2$, then $L=\{y+2,y+3\}$ or $L=\{y+2,y+4\}$ with $y \in \mathbb{N}_0$.
Both  sets are in $\mathcal{L}(C_2^5)$ by Lemma \ref{3.2}. If $|L|=3$, then $L=\{y+2,y+3,y+4\}$, or
$L=\{y+2,y+4,y+5\}$, or $L=\{y+3,y+4,y+6\}$ with $y \in \mathbb{N}_0$; recall that $\{2,3,5\} \notin \mathcal{L}(C_6)$ by Lemma \ref{3.1}.
All these sets are in $\mathcal{L}(C_2^5)$ by Lemma \ref{3.2}.
If $|L|\ge 4$, then $L$ has one of the forms given in Proposition \ref{AMP4_C6}. All these sets are in $\mathcal{L}(C_2^5)$ by Lemmas \ref{realC25_1}, \ref{realC25_2}, and \ref{realC25_3}.
\end{proof}

\subsection{On $\mathcal L (G_0) \subset \mathcal L (C_2^5)$ for some subsets $G_0 \subset C_6$}  \label{3.d}

The goal of this subsection is to prove that $\mathcal L (G_0) \subset \mathcal L (C_2^5)$ for several subsets $G_0$ of a cyclic group $G$ of order $|G|=6$.
The first lemma is of interest in its own rights. It shows that -- in contrast to the expected affirmative answer to the Characterization Problem -- groups $G$ may have proper subgroups $G_0$ such that $\mathcal L (G_0)$ and further arithmetical invariants are equal to the invariants  of a different group $G'$.

\begin{lemma} \label{7.1}
Let $G$ be a finite abelian group, $g \in G$ with $\ord (g)=6$ and $G_0 = \{0, g, 2g, 3g, 4g \}$. Then $\mathsf c (G_0)=3$, $\rho (G_0)=3/2$, and
\[
\mathcal L (G_0) = \{ y + 2k + [0,k] \colon y, k \in \N_0 \} = \mathcal L (C_3) = \mathcal L (C_2 \oplus C_2) \subset \mathcal L (C_2^5) \,.
\]
\end{lemma}

\begin{proof}
Theorem {\bf A} (stated in the Introduction) implies that $\mathcal L (C_3)= \mathcal L (C_2 \oplus C_2)$ has the given form and, clearly, $\mathcal L (C_2 \oplus C_2) \subset \mathcal L (C_2^5)$. Thus it remains to show that $\mathcal L (G_0)$ has the indicated form.
We set $G_1 = \{0, g, 2g, 4g\}$ and proceed in two steps.

1. Since $G_2 = \{0, 2g, 4g\}$ is a cyclic group of order three, we obtain that $\mathsf c (G_2) = 3$, $\rho (G_2)=3/2$, and
\[
\mathcal L (G_2) = \{ y + 2k + [0,k] \colon y, k \in \N_0 \} = \mathcal L (C_3) = \mathcal L (C_2 \oplus C_2) \subset \mathcal L (C_2^5) \,.
\]
For every $B \in \mathcal B (G_1)$ the multiplicity $\mathsf v_g (B)$ is even. Thus the homomorphism $\theta\colon \mathcal B (G_1) \to \mathcal B (G_2)$, defined by $\theta (B) = g^{-\mathsf v_g (B)}(2g)^{\mathsf v_g (B)/2}B$, is a transfer homomorphism. This implies that $\mathsf c (G_1)=\mathsf c (G_2)$, $\rho (G_1)=\rho (G_2)=3/2$, and $\mathcal L (G_1) = \mathcal L (G_2)$ (for background on transfer homomorphism we refer to \cite[Section 3.2]{Ge-HK06a}).

2. Since $\mathcal B (G_1)$ is a divisor-closed submonoid of $\mathcal B (G_0)$, it follows that $3 = \mathsf c (G_1) \le \mathsf c (G_0)$, $3/2 = \rho (G_1) \le \rho (G_0)$, and $\mathcal L (G_1) \subset \mathcal L (G_0)$. There are precisely four atoms containing $3g$, namely
\[
U_0 = (3g)^2, \ U_1 = g^3(3g), \ U_2 = g(2g)(3g), \ \text{and} \ U_3 = (4g)^2(3g)g \,.
\]
We continue with the following two assertions. Let $A \in \mathcal B (G_0)$.
\begin{enumerate}
\item[{\bf A1.}\,] $\mathsf c (A) \le 3$

\item[{\bf A2.}\,] $\rho ( \mathsf L (A)) \le 3/2$.
\end{enumerate}
Suppose that {\bf A1} and {\bf A2} hold. Then $\mathsf c (G_0)=3$, $\rho (G_0)=3/2$, and $\mathcal L (G_0) = \mathcal L (G_1)$.

\smallskip
{\it Proof of \,{\bf A1}}.\, There is a $3$-chain of factorizations from any factorization $z \in \mathsf Z (A)$ to a factorization $z^* \in \mathsf Z (A)$ where $\mathsf v_{U_0} (z^*)$ is maximal, say $z^* = z_1 U_0^m$ and $A = A_1U_0^m$. First, we suppose that $\mathsf v_{3g}(A)$ is even. Then $m = \mathsf v_{3g}(A)/2$, and $A_1 \in \mathcal B (G_1)$. Since $\mathsf c (A_1) \le 3$, any two factorizations of $A$, which are divisible by $U_0^m$, can be concatenated by a $3$-chain of factorizations, whence $\mathsf c (A) \le 3$. If $\mathsf v_{3g}(A)$ is odd, then there is an $i \in [1,3]$ with $\mathsf v_{U_i} (z_1)=1$ and $\mathsf v_{U_j} (z_1)=0$ for $j \in [1,3] \setminus \{i\}$. Arguing as above we infer that $\mathsf c (A) \le 3$.

\smallskip
{\it Proof of \,{\bf A2}}.\, In order  to verify that $\rho ( \mathsf L (A)) \le 3/2$, we have to show that, for any two lengths $m_1, m_2 \in \mathsf L (A)$, we have $m_2/m_1 \le 3/2$. There are precisely two atoms with $g$-norm greater than one. These are $U_3$ and $U_4 = (4g)^3$ and we have $\| U_3 \|_g = \| U_4 \|_g = 2$. When in a $3$-chain of factorizations the length increases, then the number of atoms with $g$-norm two decreases by one and the number of atoms with $g$-norm one increases by two. Let $z_1$ and $z_2$ be factorizations of length $m_1$ and $m_2$ and consider a $3$-chain of factorizations from $z_1$ to $z_2$. Suppose there are $k$ atoms with $g$-norm one and $\ell$ atoms with $g$-norm two in the factorization $z_1$. It follows that in the factorization $z_2$ there are $k-s$ atoms with $g$-norm one and $\ell + 2s$ atoms with $g$-norm two for some $s \in [0, \min \{k, \ell \}]$. Then
\[
\frac{m_2}{m_1} = \frac{k+\ell+s}{k+\ell} \le \frac{3}{2} \,. \qedhere
\]
\end{proof}

\begin{lemma} \label{7.2}
Let $G$ be a finite abelian group, $g \in G$ with $\ord (g)=6$ and $G_0 = \{0, g,  -g \}$. Then
\[
\mathcal L (G_0) = \{ y + 2k + 3 \cdot [0,k] \colon y, k \in \N_0 \} \subset \mathcal L (C_2^5) \,.
\]
\end{lemma}

\begin{proof}
Let $A \in \mathcal B (G_0)$. Without restriction we may suppose that $\mathsf v_g (A) \ge \mathsf v_{-g}(A)$. Thus there are $j \in [0,4]$, $k, \ell, m \in \N_0$ such that
\[
A = g^{5k+j}(-g)^{5k+j}0^m g^{5\ell} \,.
\]
Thus
\[
\mathsf L (A) = j+\ell+m + \mathsf L ( g^{5k}(-g)^{5k} = j + \ell + m + 2k + 3 \cdot [0,k] \,,
\]
whence the claim follows.
\end{proof}

\begin{lemma} \label{7.3}
Let $G$ be a finite abelian group, $g \in G$ with $\ord (g) = 6$, and   $G_0 = \{0, g, 3g, -g\}$. Then $\mathcal L (G_0) \subset \mathcal L (C_2^5)$.
\end{lemma}

\begin{proof}
Let $A \in \mathcal B (G_0)$. By Lemmas \ref{7.1} and \ref{7.2}, we may suppose that $\mathsf v_g (A)$, $\mathsf v_{-g}(A)$, and $\mathsf v_{3g} (A)$ are positive.
Without restriction we may suppose that $\mathsf v_0 (A)=0$ and $\mathsf v_g (A) \ge \mathsf v_{-g} (A)$. Then $A$ can be written in the form
\[
A = \big( (-g)g \big)^{6r+s} \big( (3g)^2 \big)^t \big( (3g)g^3 \big)^u \big( g^6 \big)^v \,,
\]
where $r,t, v \in \N_0$, $s \in [0,5]$, and $u \in [0,1]$. Since all atoms in the decomposition have $g$-norm one, it follows that $\max \mathsf L (A) = 6r+s+t+u+v$. Thus, we obtain that
\[
\mathsf L (A) = u+v+\mathsf L (A_1), \quad \text{where} \quad A_1 = \big( (-g)g \big)^{6r+s} \big( (3g)^2 \big)^t \,.
\]
Defining
\[
A_2 = U^{2r} \big( (e_1+e_2+e_3)(e_4+e_5+e_{[1,5]}) \big)^t \prod_{i=1}^s (e_i^2) \quad \text{where} \quad U = e_{[1,5]} e_1 \cdot \ldots \cdot e_5  \,,
\]
we infer that  $\mathsf L (A_1) = \mathsf L (A_2) \in \mathcal L (C_2^5)$.
\end{proof}

\subsection{Proof of Theorem \ref{1.1}.1} \label{3.e}
We have to consider finite abelian groups $G$ with $\mathsf D (G) \in [4,6]$. Let $m \in [4,6]$. The claims that $\mathcal L (C_m)$ is minimal in $\Omega_m$ and that $\mathcal L (C_2^{m-1})$ is maximal in $\Omega_m$ follow from \cite[Theorem 3.5]{Ge-Sc-Zh17b}. It remains to verify that $\mathcal L (C_m) \subsetneq \mathcal L (C_2^{m-1})$. Since $\mathcal L (C_m) \ne \mathcal L (C_2^{m-1})$ by Proposition \ref{2.2}.3, it suffices to verify inclusion.

1. By \cite[Theorem 7.3.2]{Ge-HK06a}, we have
\begin{itemize}
\item $\mathcal L (C_4) = \bigl\{ y + k+1 + [0,k] \, \colon \, y,
      \,k \in \N_0 \bigr\} \,\cup\,  \bigl\{ y + 2k + 2 \cdot [0,k] \, \colon
      \, y,\, k \in \N_0 \bigr\} $,

\item $\mathcal L (C_2^3)  =  \bigl\{ y + (k+1) + [0,k] \, \colon \, y \in \N_0, \ k \in [0,2] \bigr\}$ \newline
      $\quad \text{\, } \ \qquad$ \quad $\cup \ \bigl\{ y + k + [0,k] \, \colon \, y \in \N_0, \ k \ge 3 \bigr\}
      \cup \bigl\{ y + 2k
      + 2 \cdot [0,k] \, \colon \, y ,\, k \in \N_0 \bigr\}$,
\end{itemize}
whence $\mathcal L (C_4) \subset \mathcal L (C_2^3)$.

\smallskip
2. Theorems 4.3 and 4.8 in \cite{Ge-Sc-Zh17b} provide explicit descriptions of $\mathcal L (C_5)$ and of $\mathcal L (C_2^4)$. These descriptions show that $\mathcal L (C_5) \subset \mathcal L (C_2^4)$.

\smallskip
3.  Let $G$ be a cyclic group of order $|G|=6$ and let $g \in G$ with $\ord (g)=6$. Let $A' \in \mathcal B (G)$. If $A'= 0^kA$, with $k \in \N_0$ and $A \in \mathcal B (G\setminus \{0\})$, then $\mathsf L (A') = k+\mathsf L (A)$ and it suffices to verify that $\mathsf L (A) \in \mathcal L (C_2^5)$.  If $\{g, -g\} \not\subset \supp (A)$, say $-g \not\in G_0$, then $\supp (A) \subset \{g,2g,3g,4g \}$, whence Lemma \ref{7.1} implies that $\mathsf L (A) \in \mathcal L (C_2^5)$. Thus from now on we suppose that $\{g, -g \} \subset \supp (A)$.  If $\supp (A) \subset \{g,3g,-g\}$, then $\mathsf L (A) \in \mathcal L (C_2^5)$ by Lemma \ref{7.3}.  Thus it remains to consider the following four cases.

\smallskip
\noindent
CASE 1: $\supp (A) = \{g,2g,-g\}$ or $\supp (A) = \{g, 4g, -g\}$.

Then \cite[Lemma 3.6]{Ge-Sc19d} implies that $\mathsf L (A)$ is either an interval or an AMP with periods $\{0,1,4\}$, $\{0,3,4\}$, $\{0,1,2,4\}$, $\{0,1,3,4\}$, or $\{0,2,3,4\}$,  and all these cases actually occur. Thus $\mathsf L (A) \in \mathcal L (C_2^5)$ by Propositions \ref{3.3}, \ref{3.5}, \ref{3.6}, and Proposition \ref{3.11}.

\smallskip
\noindent
CASE 2: $\supp (A) = \{g,2g,4g, -g\}$.

Then \cite[Lemma 3.7]{Ge-Sc19d} implies that $\mathsf L (A)$ is either an interval or an AMP with periods  $\{0,1,2,4\}$, $\{0,1,3,4\}$, or $\{0,2,3,4\}$, and all these cases actually occur. Thus $\mathsf L (A) \in \mathcal L (C_2^5)$ by Proposition \ref{3.3} and Proposition \ref{3.c}.

\smallskip
\noindent
CASE 3: $\supp (A) = \{g,2g,3g,-g\}$ or $\supp (A) = \{g, 3g, 4g, -g\}$.

Then \cite[Lemma 3.3]{Ge-Sc19d} shows that $\mathsf L (A)$ is  an interval, whence $\mathsf L (A) \in \mathcal L (C_2^5)$ by Proposition \ref{3.3}.

\smallskip
\noindent
CASE 4: $\supp (A) = G \setminus \{0\}$.

Then $\mathsf L (A)$ is an interval by \cite[Theorem 7.6.9]{Ge-HK06a}, whence $\mathsf L (A) \in \mathcal L (C_2^5)$ by Proposition \ref{3.3}. \qed

\section{Proof of Theorem \ref{1.1}.2} \label{4}

The goal in this section is to prove Theorem \ref{1.1}.2.
We start with three lemmas.

\begin{lemma} \label{4.1}
Let $G$ be a finite abelian group with $|G| \ge 3$.
\begin{enumerate}
\item The following statements are equivalent.
      \begin{enumerate}
      \item[(a)] Every $L \in \mathcal L (G)$ with $\{2, \mathsf D (G)\} \subset L$ satisfies \ $L = \{2, \mathsf D (G) \}$.

      \item[(b)] $\{ 2, \mathsf  D (G) \} \in \mathcal L (G)$.

      \item[(c)] $G$ is either cyclic or an elementary $2$-group.
      \end{enumerate}

\item We have $\daleth (G) \le  2+ \max \Delta (G) \le \mathsf c (G) \le \mathsf D (G)$,  and $\daleth (G) = \mathsf D (G)$  if and only if $G$ is either cyclic or an elementary $2$-group.
\end{enumerate}
\end{lemma}

\begin{proof}
1. See \cite[Theorem 6.6.3]{Ge-HK06a}.

2. The chain of inequalities was already observed in \eqref{basic inequalities}. If $G$ is cyclic or an elementary $2$-group, then 1. implies that $\daleth (G) = \mathsf D (G)$. The reverse implication follows from \cite[Theorem 6.4.7]{Ge-HK06a}.
\end{proof}

\begin{lemma} \label{4.2}
Let $G$ be a cyclic group of order $|G|=n \ge 4$. Then $\{2, n-2, n-1\} \in \mathcal L (G)$.
\end{lemma}

\begin{proof}
Let $g \in G$ with $\ord (g)=n$. Then $U = g^{n-2}(2g) \in \mathcal A (G)$ and $\mathsf L \big( U(-U) \big) = \{2, n-2, n-1 \}$.
\end{proof}

\begin{lemma} \label{4.3}
Let $G = C_2^r$ with $r \ge 3$. Then $\{2, r-1, r\} \in \mathcal L (G)$ if and only if $r \in [3,5]$.
\end{lemma}

\begin{proof}
We have $\{2,3\} \in \mathcal L (C_2^2) \subset \mathcal L (C_2^3)$ by Theorem {\bf A}. We have $\{2,3,4\} \in \mathcal L (C_2^4)$ by \cite[Theorem 4.8]{Ge-Sc-Zh17b}, and $\{2,4,5\} \in \mathcal L (C_2^5)$ by Lemma \ref{3.2}.1.

Suppose that $r \ge 6$ and
assume to the contrary that there are $U,V \in \mathcal A (G)$ such that $\mathsf L (UV) = \{2, r-1, r\}$. Without restriction we may suppose that $|U| \ge |V|$. Clearly, we have $|V| \ge r$. We distinguish three  cases.

\smallskip
\noindent
CASE 1: \ $|U|=|V|=r$.

Then $V = -U$ and $\langle \supp (U) \rangle \cong C_2^{r-1}$. Since $\mathsf D (C_2^{r-1}) = r$ and $\{2, r \} \subset \mathsf L (UV)$, Lemma \ref{4.1}.1 implies that $\mathsf L (UV)= \{2, r\}$, a contradiction.

\smallskip
\noindent
CASE 2: \ $|U|=r+1$ and $|V|=r$.

Then there are $W_0, \ldots, W_{r-1} \in \mathcal A (G)$ such that $UV = W_0W_1 \cdot \ldots \cdot W_{r-1}$, where $V = e_0 \cdot \ldots \cdot e_{r-1}$ and   $W_i = e_i^2$ for all $i \in [1, r-1]$, $|W_0|=3$, and $e_0 \mid W_0$. Since $\langle \supp (U) \rangle \cong C_2^r$, there is $e_r \in G$ such that $(e_1, \ldots, e_r)$ is a basis of $G$, $U = (e_0+e_r)e_1 \cdot \ldots \cdot e_r$, and $W = (e_0+e_r)e_0e_r$. This implies that  $\mathsf L (UV)= \{2, r\}$, a contradiction.

\smallskip
\noindent
CASE 3: \ $|U|=|V|=r+1$.

We set $UV = W_1 \cdot \ldots \cdot W_r$, where $W_1, \ldots, W_r \in \mathcal A (G)$ with $2\le |W_1| \le \ldots \le |W_r|$. There are the following  two cases.

\smallskip
\noindent
CASE 3.1: \ $|W_1|= \ldots = |W_{r-2}|=2$ and $|W_{r-1}|=|W_r|=3$.

We set $W_i = e_i^2$ for $i \in [1, r-2]$,   $W_{r-1} = e_{r-1}e_r(e_{r-1}+e_r)$,  and  $U = e_1 \cdot \ldots \cdot e_r e_0$. Then $V$ has the form
\[
V = e_1 \cdot \ldots \cdot e_{r-2} (e_{r-1}+e_r) g (e_0+g) \quad \text{for some} \ g \in G \,.
\]
Clearly, $V' = \Big( e_1 \cdot \ldots \cdot e_{r-2} (e_{r-1}+e_r) e_0 \Big) \in \mathcal A (G)$. Since $e_1 \cdot \ldots \cdot e_r$ is zero-sum free, $U' = e_1 \cdot \ldots \cdot e_r g (e_0+g)$ is a product of two atoms. Thus $UV = U'V'$ has a factorization of length three, a contradiction because $3 < r-1$.

\smallskip
\noindent
CASE 3.2: \ $|W_1|= \ldots = |W_{r-1}|=2$ and $|W_r|=4$.

We set $W_i = e_i^2$ for $i \in [1, r-1]$,  $U = e_1 \cdot \ldots \cdot e_{r-1}e_r e_0$, and $W_r = e_re_0 e_r'e_0'$. Then
\[
V = e_1 \cdot \ldots \cdot e_{r-1}e_r' e_0' \,.
\]
Then $e_0'+e_r'=e_1+ \ldots + e_{r-1}$, whence $e_r' = e_r+g$ and $e_0' = e_0+g$ for some $g \in G$. Since $W_r \in \mathcal A (G)$, it follows that $g \ne 0$. Since $(e_1, \ldots, e_r)$ is a basis of $G$ and $e_r' \notin \langle e_1, \ldots , e_{r-1} \rangle$, it follows that $g= e_r'+e_r \in \langle e_1, \ldots , e_{r-1} \rangle$. Thus $g= e_I = \sum_{i \in I} e_i$ with $\emptyset \ne I \subset [1, r-1]$. If $I = [1,r-1]$, then $e_r' = e_0$, a contradiction to $W_r \in \mathcal A (G)$. Thus $UV$ has a factorization
\[
UV = \Big( e_r (e_r+e_I)\prod_{i \in I} e_i \Big) \Big( e_0 (e_0+e_I)\prod_{i \in I} e_i \Big) \prod_{i \in [1, r-1]\setminus I} e_i^2
\]
of length $2+(r-1-|I|)=r+1-|I|$ and a factorization
\[
UV = \Big( e_0(e_r+e_I)\prod_{i \in [1, r-1]\setminus I} e_i \Big) \Big( e_r(e_0+e_I)\prod_{i \in [1, r-1]\setminus I} e_i \Big) \prod_{i \in I} e_i^2
\]
of length $2 + |I|$. If $2+|I|=r$, then $r+1-|I|=3 < r-1$, a contradiction. If $2+|I|=r-1$, then $r+1-|I|= 4 < r-1$, a contradiction.
\end{proof}

\medskip
\begin{proof}[Proof of Theorem \ref{1.1}.2]
Let $m \ge 7$.
By \cite[Theorem 3.5]{Ge-Sc-Zh17b},  $\mathcal L (C_2^{m-1})$ is a maximal element of $\Omega_m$,  $\mathcal L (C_m)$ is a minimal element of $\Omega_m$, and if $G$ is an abelian group with $\mathsf D (G)=m$ and $\mathcal L (G) \subset \mathcal L (C_2^{m-1})$, then $G$ is either cyclic or an elementary $2$-group. Thus it remains to prove the following two assertions.

\begin{enumerate}
\item[{\bf A1.}\,] $\mathcal L (C_m)$ is not contained in $\mathcal L (C_2^{m-1})$.

\item[{\bf A2.}\,] If $G$ is a finite abelian group with $\mathsf D (G)=m$ and $\mathcal L (C_m) \subset \mathcal L (G)$, then $G$ is cyclic of order $m$.
\end{enumerate}

\smallskip
Since $m \ge 7$, {\bf A1}  follows from Lemmas \ref{4.2} and \ref{4.3}. To verify {\bf A2}, let $G$ be a finite abelian group with $\mathsf D (G)=m$ such that $\mathcal L (C_m) \subset \mathcal L (G)$. Then Lemma \ref{4.1}.1 implies that $\{2, \mathsf D (G)\} \in \mathcal L (C_m) \subset \mathcal L (G)$. Now, again Lemma \ref{4.1}.1 implies that $G$ is either cyclic or an elementary $2$-group. Finally, {\bf A1} implies that $G$ is not an elementary $2$-group, whence $G$ is cyclic of order $m$.
\end{proof}

\section{Proof of Theorem \ref{1.1}.3, \ref{1.1}.4, and of Corollary \ref{1.2}} \label{5}

The goal in this section is to prove Statements $3$ and $4$ of Theorem \ref{1.1}  and Corollary \ref{1.2}.
We need some lemmas.

\begin{lemma} \label{5.1}
Let $G$ be a finite abelian group with $\mathsf D (G) \ge 5$.
\begin{enumerate}
\item The following statements are equivalent.
      \begin{enumerate}
      \item $G$ is isomorphic to $C_2 \oplus C_{2n}$ with $n \ge 2$.

      \item $\{2, \mathsf D (G)-1, \mathsf D (G)\} \in \mathcal L (G)$.
      \end{enumerate}

\item  The following statements are equivalent.
      \begin{enumerate}
      \item $\daleth (G) = \mathsf D (G)-1$.

      \item $G$ is isomorphic either to $C_2^{r-1} \oplus C_4$ for some $r \ge 2$ or to $C_2 \oplus C_{2n}$ for some $n \ge 2$.
      \end{enumerate}
\end{enumerate}
\end{lemma}

\begin{proof}
See \cite[Theorem 1.1 and Proposition 3.5]{Ge-Zh15b}.
\end{proof}

\begin{lemma} \label{5.2}
Let $G = C_2 \oplus C_{2n}$ with $n \ge 2$.
A sequence $S$ over $G$ of length $\mathsf D (G) = 2n+1$ is a minimal zero-sum sequence if and only if it has one of the following two forms.
      \begin{itemize}
      \item[(a)] $S = g^{2n-1}h (g-h)$ for some $g \in G$ with $\ord (g) = 2n$ and some $h \in G \setminus \langle g \rangle$.

      \item[(b)] $S = e g^v  (g+e)^{2n-v}$ for some $g \in G$ with $\ord (g)=2n$, $e \in G \setminus \langle g \rangle$ with $\ord (e)=2$, and $v \in [3, 2n-3]$ odd.
      \end{itemize}
\end{lemma}

\begin{proof}
See \cite[Theorem 3.3]{Ga-Ge02}.
\end{proof}

Let $G$ be a finite abelian group. Then every element $g \in G$ with $\ord (g)=\exp (G)$ can be extended to a basis of $G$. Thus in Case (a) of Lemma \ref{5.2} the element $g$ can be extended to a basis. In Case (b), $(g,e)$ and $(g+e, e)$ are bases of $G$. Let $G = C_2 \oplus C_{2n}$ with $n \ge 2$. Next we completely determine all sets $L \in \mathcal L (G)$ with $\{2, \mathsf D (G)\} \subset L$. This was done before in \cite[Lemma 3.2]{B-G-G-P13a} but,  unfortunately, that result is not correct.

\begin{proposition}\label{5.3}
Let $G = C_2\oplus C_{2n}$ with $n\ge 2$. Then
\[
\begin{aligned}
\big\{L\in \mathcal L(G)\colon \{2, \mathsf D(G)\}\subset L\}\big\} = & \ \big\{\{2, 2m, 2n-2m+2, 2n, 2n+1\}\colon m \in [1, n]\big\}  \ \cup \\
 & \ \big\{ \{ 2, 2n-2i, 2n+1-2i \colon i \in [0, (v-1)/2] \} \colon v \in [3, 2n-3] \ \text{\rm odd} \big\} \,.
\end{aligned}
\]
\end{proposition}

\begin{proof}
Let $L\in \mathcal L(G)$ with $\{2, \mathsf D(G)\}\subset L$. Then there exists an atom $U\in \mathcal A(G)$ with $|U|=\mathsf D(G)$ such that $\mathsf L(U(-U))=L$.  According to the structure  of $U$, as given in  Lemma \ref{5.2}, we distinguish two cases.

\smallskip
\noindent
CASE 1: There exists a basis $(e_1, e_2)$ of $G$ with $\ord(e_1)=2n$ and $\ord(e_2)=2$ such that
\[
U=e_1^{2n-1}(x_1e_1+e_2)(x_2e_1+e_2), \quad \text{where} \quad x_1, x_2\in [0, 2n-1] \quad \text{ with} \quad x_1+x_2\equiv 1 \mod {2n} \,.
\]
Note that the congruence condition on $x_1$ and $x_2$ implies that $x_1 \ne x_2$ and that $|x_1-x_2|$ is odd. By symmetry, we may suppose that $x_1 > x_2$. We consider a factorization $z \in \mathsf Z \big( U(-U) \big)$ of length $|z| \in [2, 2n]$.
Let $W \in \mathcal A (G)$  be an atom occurring in the factorization $z$ and with  $x_1e_1+e_2\in \supp(W)$. Since $|W|=2$ would imply that $|z|=2n+1$, it follows that $|W| \in [3, \mathsf D (G)]$.
If $x_2e_1+e_2\in \supp(W)$, then
\[
W=U \ \text{ and } \ |z|=2 \qquad \text{ or } \qquad W=(-e_1)(x_1e_1+e_2)(x_2e_1+e_2) \ \text{ and } \ |z|=2n \,.
\]
Suppose $-x_2e_1+e_2\in \supp(W)$.  Then
\[
W=(-e_1)^{x_1-x_2}(x_1e_1+e_2)(-x_2e_1+e_2) \quad \text{and} \quad z = W (-W) \big( e_1(-e_1) \big)^{2n-1-(x_1-x_2)}
\]
or
\[
W=e_1^{2n-(x_1-x_2)}(x_1e_1+e_2)(-x_2e_1+e_2) \quad \text{and} \quad z = W (-W) \big( e_1(-e_1) \big)^{(x_1-x_2)-1} \,.
\]
This implies that
\[
\mathsf L(U(-U))=\{ 2, (x_1-x_2)+1, 2n +1 -(x_1-x_2), 2n, 2n+1\}\,.
\]
We set $x_1-x_2 = 2m-1$ and note that all values $m \in [1,n]$ can occur.

\smallskip
\noindent
CASE 2: There exists a basis $(e_1, e_2)$ of $G$ with $\ord(e_1)=2$ and $\ord(e_2)=2n$ such that
\[
U=e_1e_2^v(e_1+e_2)^{2n-v} , \quad \text{ where } \quad v\in [3, 2n-3] \quad  \text{ odd}  \,.
\]
Without restriction we may assume that $v \le  2n-v$. We list the atoms of $\mathcal A (G)$ which divide $U(-U)$:
\begin{itemize}
\item $U, -U$, \ $e_1^2, e_2(-e_2)$, and $(e_1+e_2)(e_1-e_2)$.

\item $(e_1+e_2)(-e_2)e_1$ and $(e_1-e_2)e_2e_1$.

\item $(e_1+e_2)^2(-e_2)^2$ and $(e_1-e_2)^2e_2^2$.
\end{itemize}
We set $W = (e_1+e_2)(-e_2)e_1$ and consider a factorization $z \in \mathsf Z \big( U(-U) \big)$ of length $|z| > 2$. There are precisely the following three types of factorizations.

\smallskip
\noindent
CASE 2.1: The atom $e_1^2$ divides $z$.

Then
\[
z = (e_1^2) \Big( e_2^2(e_1-e_2)^2 \Big)^i \Big( (-e_2)^2 (e_1+e_2)^2 \Big)^i \Big( e_2(-e_2) \Big)^{v-2i} \Big( (e_1+e_2)(e_1-e_2) \Big)^{2n-v-2i}
\]
with $i \in [0, (v-1)/2]$, whence $|z| \in \{ 2n+1 -2i \colon i \in [0, (v-1)/2] \}$.

\smallskip
\noindent
CASE 2.2: $W(-W)$ divides $z$.

Then
\[
z = W (-W) \Big( e_2^2(e_1-e_2)^2 \Big)^i \Big( (-e_2)^2 (e_1+e_2)^2 \Big)^i \Big( e_2(-e_2) \Big)^{v-1-2i} \Big( (e_1+e_2)(e_1-e_2) \Big)^{2n-1-v-2i}
\]
with $i \in [0, (v-1)/2]$, whence $|z| \in \{ 2n -2i \colon i \in [0, (v-1)/2] \}$.

\smallskip
\noindent
CASE 2.3: $W^2$ divides $z$ or $(-W)^2$ divides $z$.

We may assume without restriction that $W^2$ divides $z$. Then
\[
z = W^2  \Big( e_2^2(e_1-e_2)^2 \Big)^{i+1} \Big( (-e_2)^2 (e_1+e_2)^2 \Big)^i \Big( e_2(-e_2) \Big)^{v-2-2i} \Big( (e_1+e_2)(e_1-e_2) \Big)^{2n-2-v-2i}
\]
with $i \in [0, (v-3)/2]$, whence $|z| \in \{ 2n -1-2i \colon i \in [0, (v-3)/2] \}$. Putting all together we infer that
\[
\mathsf L \big( U(-U) \big) = \{ 2, 2n-2i, 2n+1-2i \colon i \in [0, (v-1)/2] \} \,. \qedhere
\]
\end{proof}

\begin{lemma} \label{5.4}
Let $G = C_3^3$ and let $(e_1,e_2,e_3)$ be a basis of $G$. If $e_0=e_1+e_2+e_3$ and $U = e_1^2e_2^2e_3^2e_0$, then
\[
\mathsf L (U^{3k}) = 3k + 2 \cdot [0, 2k] \quad \text{for every $k \in \N$} \,.
\]
In particular, $\rho \big( \mathsf L (U^{3k}) \big) = 7/3$ for every $k \in \N$.
\end{lemma}

\begin{proof}
We set $G_0 = \{e_0,e_1,e_2,e_3\}$, $W = e_1e_2e_3e_0^2$, and $V_i=e_i^3$ for each  $i \in [0,3]$ Then $\mathcal A (G_0) = \{V_0,V_1,V_2,V_3, U, W \}$ and $\Delta (G_0) = \{2\}$, whence the assertion follows.
\end{proof}

A subset $G_0 \subset G$ is called an {\rm LCN}-set if for every $A = g_1 \cdot \ldots \cdot g_{\ell} \in \mathcal A (G_0)$ the cross number $\mathsf k (A) = \sum_{i=1}^{\ell} \frac{1}{\ord (g_i)} \ge 1$ holds. We set
\[
\mathsf m (G) = \max \{\min \Delta (G_0) \colon G_0 \subset G \ \text{is a non-half-factorial LCN-set} \} \,.
\]
For $d \in \N$, $M \in \N_0$, and $\{0,d\} \subset \mathcal D \subset [0,d]$, let $\mathcal P_M ( \mathcal D, G)$ denote the set of all $B \in \mathcal B (G)$ with $\mathsf L (B)$ is an AAMP with period $\mathcal D$ and bound $M$.

\begin{lemma} \label{5.5}
Let $G = C_4 \oplus C_4$. Then, for every sufficiently large $M$,
\[
\limsup_{B \in \mathcal P_M ( \{0,2 \}, G), \min \mathsf L (B) \to \infty} \rho (\mathsf L (B)) = 2 \,.
\]
\end{lemma}

\begin{proof}
We have $\max \Delta^* (G) =2$ by Proposition \ref{2.2}.2 and $\mathsf m (G) = 1$ by \cite[Proposition 3.6]{Sc09c}.  For a sufficiently large $M$, we consider $\mathcal P = \mathcal P_M ( \{0,2\}, G)$  and \cite[Proposition 8.7]{Sc09b} implies that
\[
\limsup_{B \in \mathcal P, \min \mathsf L (B) \to \infty} \rho (\mathsf L (B)) \le \max \{ \rho (G_0) \colon G_0 \subset G, 2 \mid \min \Delta (G_0) \} \,.
\]
Let $G_0 \subset G$ with $2 \mid \min \Delta (G_0)$. Since $\max \Delta (G)=3$ by \cite[Lemma 3.3]{Ge-Zh15b}, it follows that $\min \Delta (G_0)=2$, whence $\min \Delta (G_0)= \max \Delta^* (G)$. Now \cite[Theorem 7.7]{Sc09b} implies that $G_0 = \cup_{i=1}^s G_i$, where $\langle G_0 \rangle = \sum_{i=1}^s \langle G_i \rangle$ and each $G_i$ is either half-factorial or  equal to $\{g_i, -g_i \}$ for some $g_i$ with $\ord (g_i)=4$. Thus $\rho (G_0)= 2$.
\end{proof}

\begin{proof}[Proof of Theorem \ref{1.1}.3]
Let $m \ge 5$. We have to show that $\mathcal L (C_2^{m-4} \oplus C_4)$ is a maximal element in $\Omega_m$.  Let $G$ be a finite abelian group with $\mathsf D (G) = m$ and suppose that $\mathcal L (C_2^r \oplus C_4) \subset \mathcal L (G)$, where  $r = m-4$. We distinguish two cases and  use Proposition \ref{2.3} without further mention.

\smallskip
\noindent
CASE 1: \, $m \in [5,6]$.

First suppose that $m=5$. Then $r=1$ and $G$ is isomorphic to one of the following groups:
\[
C_3 \oplus C_3, \ C_5, \ C_2 \oplus C_4, \ C_2^4 \,.
\]
Since $\mathcal L (C_m)$ is minimal in $\Omega_m$ by Theorem \ref{1.1}.1, it follows that $G$ is not cyclic.
Since, by Proposition \ref{2.2}, $\max \Delta^* (C_3 \oplus C_3) = 1 < 2 = \max \Delta^* (C_2 \oplus C_4)$, it follows that $G$ is not isomorphic to $C_3 \oplus C_3$. Theorems 4.5 and 4.8 in \cite{Ge-Sc-Zh17b} show that $G$ is not isomorphic to $C_2^4$.

Next suppose that $m=6$. Then $r=2$ and $G$ is isomorphic to one of the following groups:
\[
C_6, \ C_2^2 \oplus C_4, \ C_2^5 \,.
\]
Since $G$ is not cyclic, it remains to show that $\mathcal L (C_2^2 \oplus C_4) \not\subset \mathcal L (C_2^5)$. Lemma \ref{5.1}.1 implies that $\{2, \mathsf D (G)-1, \mathsf D (G)\} \in \mathcal L (C_2^2 \oplus C_4)$. Since $\{2, \mathsf D (G)-1, \mathsf D (G)\} \not\in \mathcal L (C_2^5)$ by Lemma \ref{4.1}.1, it follows that $\mathcal L (C_2^2 \oplus C_4) \not\subset \mathcal L (C_2^5)$.

\smallskip
\noindent
CASE 2: \, $m \ge 7$.

Then Theorem \ref{1.1}.2 implies that $G$ is neither cyclic nor an elementary $2$-group. Thus Lemma \ref{4.1}.2 implies that $\daleth (G) < \mathsf D (G)$. Since $\mathsf D (G)-1 = \daleth (C_2^r \oplus C_4) \le \daleth (G)$ by Lemma \ref{5.1}.2, it follows that $\daleth (G)=\mathsf D (G)-1$. It is again Lemma \ref{5.1}.2 that  implies that $G$ is either isomorphic to $C_2 \oplus C_{2n}$ or isomorphic to $C_2^s \oplus C_4$, where $n \in \N$ with $m=\mathsf D (G)=2n+1$ and $s \in \N$ with $m=\mathsf D (G)=s+4$.

Thus it remains to consider the case where $m$ is odd and to prove the following  assertion.

\begin{enumerate}
\item[{\bf A.}\,]  $\mathcal L (C_2^r \oplus C_4) \not\subset \mathcal L (C_2 \oplus C_{2n})$, with $n= (m-1)/2$.
\end{enumerate}

\smallskip

{\it Proof of \,{\bf A}}.\, Assume to the contrary that $\mathcal L (C_2^r \oplus C_4) \subset \mathcal L (C_2 \oplus C_{2n})$. Then $\rho_3 (C_2^r \oplus C_4) \le \rho_3 (C_2 \oplus C_{2n})$.
Since $n \ge 3$, \cite[Theorem 5.1]{Ge-Gr-Yu15} implies that $\rho_3 (C_2 \oplus C_{2n}) < \mathsf D (G) + \lfloor \mathsf D (G)/2 \rfloor$. We claim that $\rho_3 (C_2^r \oplus C_4) = \mathsf D (G) + \lfloor \mathsf D (G)/2 \rfloor$, which yields a contradiction and ends the proof.
Since $2n+1=m=r+4 \ge 7$, there is $s \in \N_0$ such that $r \in \{2s+2, 2s+3\}$. We distinguish two cases.

Suppose that $r=2s+2$. We set $G = C_2^r \oplus C_4$, $G_1=C_2^{s+2}$, and $G_2=C_2^s \oplus C_4$. Then $\mathsf d (G) = \mathsf d (G_1)+ \mathsf d (G_2)$ and $\mathsf d (G_2)= s+3=\mathsf d (G_1)+1$. Thus \cite[Theorem 6.3.4.1]{Ge-HK06a} implies that $\rho_3 (G) = \mathsf D (G) + \lfloor \mathsf D (G)/2 \rfloor$.

Suppose that $r=2s+3$. We set $G = C_2^r \oplus C_4$, $G_1=C_2^{s+3}$, and $G_2=C_2^s \oplus C_4$. Then $\mathsf d (G) = \mathsf d (G_1)+ \mathsf d (G_2)$ and $\mathsf d (G_2)= s+3=\mathsf d (G_1)$. Thus \cite[Theorem 6.3.4.1]{Ge-HK06a} implies that $\rho_3 (G) = \mathsf D (G) + \lfloor \mathsf D (G)/2 \rfloor$.
\end{proof}

\begin{proof}[Proof of Theorem \ref{1.1}.4: First part]
Let $n \ge 2$.  We show that $\mathcal L ( C_2 \oplus C_{2n})$ is a maximal element of $\Omega_{2n+1}$.
Let $G$ be a finite abelian group with $\mathsf D (G) = 2n+1$ and suppose that $\mathcal L (C_2 \oplus C_{2n}) \subset \mathcal L (G)$.
Then Lemma \ref{5.1}.1 implies that $\{2, \mathsf D (G)-1, \mathsf D (G)\} \in \mathcal L (C_2 \oplus C_{2n}) \subset \mathcal L (G)$. Now, again Lemma \ref{5.1}.1 implies that $G \cong C_2 \oplus C_{2n}$.
\end{proof}

Let $G = C_{n_1 } \oplus \ldots \oplus C_{n_r}$, with $1 < n_1 \mid \ldots \mid n_r$ and $r \ge 3$, be a finite abelian group. If $n_{r-1} \ge 3$, then \cite[Theorem 4.2]{B-G-G-P13a} implies that
\[
[2, \mathsf D^* (G)] \ \subset \ \bigcup_{L \in \mathcal L (G), \{2, \mathsf D^* (G)\} \subset L} \ L \,.
\]
The next lemma, which is needed in the proof of Corollary \ref{1.2}, shows that the above result does not hold without the assumption that $n_{r-1} \ge 3$.

\begin{lemma} \label{5.6}
Let $G = C_2^3 \oplus C_4$ and $U \in \mathcal A (G)$ with $|U|=\mathsf D (G)=7$. Then $3 \notin \mathsf L \big( U(-U) \big)$.
\end{lemma}

\begin{proof}
We start with the following simple observations.
\begin{itemize}
\item The sum of any two elements of $G$ of order four has order two.

\item If $W \in \mathcal A (G)$, the number of elements of order four in $W$ (counted with multiplicity) is even.

\item Since $\mathsf D (C_2^4)=5$, $U$ cannot have five elements (counted with multiplicity) of order two.
\end{itemize}
Thus, the number of elements of order four in $U$ is equal to four or six. We set $U = g_1 \cdot \ldots \cdot g_7$ with $\ord (g_1) \le \ldots \le \ord (g_7)$. Suppose for a contradiction that there are $W_1, W_2, W_3 \in \mathcal A (G)$ with $|W_1| \le |W_2| \le |W_3|$ such that
\[
U(-U) = W_1W_2W_3 \,.
\]
Assume to the contrary that $|W_3|=7$. This would mean that $W_3$ arises from $U$ by replacing some of the elements from $U$ by their inverses. Thus there is a subsequence $T$ of $U$ such that $W_3 = U T^{-1} (-T)$. But this implies that $W_1W_2 = (-U)(-T)^{-1}T = (-W_3)$ is an atom, a contradiction. Thus $|W_1|, |W_2|, |W_3| \in [2,6]$.
We distinguish two cases.

\smallskip
\noindent
CASE 1: \, $\ord (g_3)=2$ and $\ord (g_4) = 4$.

Then $U' = g_1g_2g_3(g_4+g_5)(g_6+g_7)$ is a minimal zero-sum sequence over an elementary $2$-group of rank four of length $5= \mathsf D (C_2^4)$. Thus there is a basis $(e_1, \ldots, e_4)$ of $G$ with $\ord (e_1) = \ord (e_2)=\ord (e_3)=2$ and $\ord (e_4)=4$ such that
$g_1=e_1, g_2=e_2, g_3 = e_3$, and $g_i = f_i + a_i e_i$ with $f_i \in \langle e_1, e_2, e_3 \rangle$ and $a_i \in \{1,3\}$ for all $i \in [1,4]$. This implies that $f_1+f_2+f_3+f_4=e_1+e_2+e_3$. Since $U$ is a minimal zero-sum sequence, it follows that $a_1= \ldots = a_4$. Without restriction we may suppose that $a_i=1$ for all $i \in [1,4]$.

Since the number of elements of order four in each atom $W_i$ is even, there is an $i \in [1,3]$ such that $W_i$ has four elements of order four, whence $(f_1 \pm e_4)(f_2 \pm e_4)(f_3 \pm e_4)(f_4 \pm e_4)$ is a subsequence of $W_1$. The only way to extend this to a zero-sum sequence is to use the elements $e_1, e_2$ and $e_3$, whence $W_i = e_1e_2e_3 (f_1 \pm e_4)(f_2 \pm e_4)(f_3 \pm e_4)(f_4 \pm e_4)$, a contradiction to $|W_i| \in [2,6]$.

\smallskip
\noindent
CASE 2: \, $\ord (g_1)=2$ and $\ord (g_2) = 4$.

We choose a basis $(e_1, \ldots, e_4)$ of $G$ with $\ord (e_1)=\ord (e_2)=\ord (e_3)=2$ and $\ord (e_4)=4$. Then $g_i = f_i + a_ie_4$ with $a_i \in \{1,3\}$ and $f_i \in \langle e_1, e_2, e_3 \rangle$ for all $i \in [2,7]$. We continue with three assertions.

\begin{enumerate}
\item[{\bf A1.}\,] $g_1 \ne 2e_4$.

\item[{\bf A2.}\,] Without restriction we may suppose that $g_1 = e_1$.

\item[{\bf A3.}\,] $|W_1|>2$.
\end{enumerate}

\smallskip
{\it Proof of \,{\bf A1}}.\, Assume to the contrary that $g_1=2e_4$. Consider the sequence $S = f_2 \cdot \ldots \cdot f_7$. If $f_2=f_3$, then $g_1g_2g_3$ is a proper zero-sum subsequence of $U$, a contradiction. Thus all elements of $S$ are pairwise distinct, whence $S$ has no zero-sum subsequence of length two.   Further,  $S$ is a zero-sum sequence over a group isomorphic to $C_2^3$. If one of the $f_i$s is equal to zero, then $f_i^{-1}S$ is still a zero-sum sequence, which is not minimal. Thus $f_i^{-1}S$ is a product of a zero-sum sequence of length two and of length three, a contradiction. Thus $S$ is a product of two minimal zero-sum sequences $S_1$ and $S_2$, and both have  length three. After renumbering if necessary, we may suppose that $S_1 = f_2f_3f_4$ and $S_2 = f_5f_6f_7$.
Since all elements of $S$ are pairwise distinct, none of the $f_i$s is equal to zero.
Since $f_4=f_2+f_3$, $f_7=f_5+f_6$ and $f_2, \ldots, f_7 \in \langle e_1, e_2, e_3 \rangle $, it follows that not all these six elements can be pairwise distinct, a contradiction.

{\it Proof of \,{\bf A2}}.\, By {\bf A1}, we have $g_1 = f$ or $g_1=f+2e_4$ with $0 \ne f \in \langle e_1, e_2, e_3 \rangle$. After renumbering if necessary, we may suppose that $f = e_1 + f'$ with $f' \in \langle e_2, e_3 \rangle$. The map $f \colon G \to G$, defined by $(e_1, e_2, e_3, e_4) \mapsto (g_1, e_2, e_3, e_4)$, is a group isomorphism. Thus there exists a basis of $G$ containing the element of $S$ having order two.

{\it Proof of \,{\bf A3}}.\, Assume to the contrary that $|W_1|=2$. Then $W_3 = -W_2$. If $\supp (W_1)$ consists of two elements of order four, then $W_2$ consists of one element of order two and five elements of order four, a contradiction to $W_2$ being a zero-sum sequence. Thus $W_1=g_1^2$. Thus $W_2$ arises from $g_1^{-1}U$ by exchanging $f_i+a_ie_4$ by $f_i-a_ie_4$ for some $i \in [2,7]$. Thus the sum of the first three coordinates of $g_1^{-1}U$ equals the  sum of the first three coordinates of $W_2$ and this is zero. Since $U$ is a zero-sum sequence and $\ord (g_1)=2$, it follows that $g_1=2e_4$, a contradiction to {\bf A1}.

By {\bf A1}, {\bf A2}, and {\bf A3},  it remains to handle the following two cases.

\smallskip
\noindent
CASE 2.1: \, $W_1 = e_1W_1'$, $W_2=e_1W_2'$, with $|W_1'|=2$ and $|W_2'|=4$, and $W_3$ consists of six elements of order four.

Since $U(-U) = W_1W_2W_3 = W_3 \Big(e_1^2 \Big) (-W_3)$, $U(-U)$ has a factorization of length three, where one atom has length two, a contradiction to {\bf A3}.

\smallskip
\noindent
CASE 2.2: \, $|W_1|=4$ and $W_i = e_1W_i'$ with $|W_i'|=4$ for $i \in [2,3]$.

Thus $W_1 \mid g_2 \cdot \ldots \cdot g_7 (-g_1) \cdot \ldots \cdot (-g_7)$. After renumbering if necessary, we infer that either
\[
W_1 = g_2g_3(-g_4)(-g_5) \quad \text{or} \quad W_1 = g_2g_3g_4(-g_5) \,.
\]
If $W_1 = g_2g_3(-g_4)(-g_5)$, then $g_2+g_3=g_4+g_5$. Since $\ord (g_2+g_3)=2$, $g_2g_3g_4g_5$ is a zero-sum subsequence of $U$, a contradiction.

Suppose that $W_1 = g_2g_3g_4(-g_5)$. After renumbering if necessary, we may assume that $\gcd (U, W_2') = g_5g_6$ and $\gcd (U, W_3')= g_7$. Thus there are $i, j \in [2,7] \setminus \{5,6\}$ such  that
\[
W_2' = (-g_i)(-g_j)g_5g_6  \,.
\]
Then $0 = \sigma (W_2) = e_1-g_i-g_j+g_5+g_6$ and thus $g_5+g_6+e_1=g_i+g_j$. Since $\ord (g_i+g_j)=2$, it follows that $e_1g_5g_6g_ig_j$ is a zero-sum subsequence of $U$, a contradiction.
\end{proof}

\begin{proof}[Proof of Corollary \ref{1.2}]
Let $G_1$ and $G_2$ be non-isomorphic finite abelian groups with $\mathsf D (G_1)=\mathsf D (G_2)=m \in [4,7]$. We need the following results.  Theorem \ref{1.1}.1 implies that for $m \in [4,6]$, $\mathcal L (C_m)$ is minimal in $\Omega_m$, $\mathcal L (C_2^{m-1})$ is maximal in $\Omega_m$, and $\mathcal L (C_m) \subsetneq \mathcal L (C_2^{m-1})$. By Theorem \ref{1.1}.2, $\mathcal L (C_7)$ and $\mathcal L (C_2^6)$ are each incomparable in $\Omega_m$. Moreover, if $G$ is an abelian group with $\mathsf D (G)=m$ and $\mathcal L (G) \subset \mathcal L (C_2^{m-1})$ for $m \ge 4$, then $G$ is either cyclic or an elementary $2$-group (\cite[Theorem 3.5]{Ge-Sc-Zh17b}).
We use Proposition \ref{2.3} without further mention.

\smallskip
\noindent
CASE 1: \, $m = 4$.

Since the only groups with Davenport constant four are the cyclic group of order four and the elementary $2$-group of rank three, the claim follows immediately from the above mentioned results.

\smallskip
\noindent
CASE 2: \, $m = 5$.

Every finite abelian group $G$ with $\mathsf D (G)=5$ is isomorphic to one of the following groups:
\[
C_5, \ C_2 \oplus C_4, \ C_3 \oplus C_3, \  C_2^4 \,.
\]
Note that $\mathcal L (C_5) \subsetneq \mathcal L (C_2^4)$,  that $\mathcal L (C_5)$ is a minimal element in $\Omega_5$, that $\mathcal L (C_2^4)$ is a maximal element in $\Omega_5$, and that the only group $G$ with $\mathsf D (G)=5$ and $\mathcal L (G) \subset \mathcal L (C_2^4)$ is cyclic of order five.

Since $\max \Delta (C_5)=3$ by Lemma \ref{4.1}.2, $\max \Delta (C_3 \oplus C_3) = 1$ by \cite[Corollary 6.4.9]{Ge-HK06a}, and $\max \Delta (C_2 \oplus C_4)=2$ by Lemma \ref{5.1}.2, it follows that $\mathcal L (C_5) \not\subset \mathcal L (C_2 \oplus C_4) \not\subset \mathcal L (C_3 \oplus C_3)$ and $\mathcal L (C_5) \not\subset \mathcal L (C_3 \oplus C_3)$. Theorems 4.1 and 4.5 in  \cite{Ge-Sc-Zh17b} show that $[2,5] \in \mathcal L (C_3 \oplus C_3) \setminus \mathcal L (C_2 \oplus C_4)$. Thus the claim follows.

\smallskip
\noindent
CASE 3: \, $m = 6$.

Every finite abelian group $G$ with $\mathsf D (G)=6$ is isomorphic to one of the following groups:
\[
C_6, \ C_2^2 \oplus C_4, \ C_2^5 \,.
\]
Again, we note that $\mathcal L (C_6) \subsetneq \mathcal L (C_2^5)$.
Since  $\mathcal L (C_6)$ is  minimal  in $\Omega_6$, and $\mathcal L (C_2^5)$ and $\mathcal L (C_2^2 \oplus C_4)$ are both maximal in $\Omega_6$,
it remains to show that $\mathcal L (C_6) \not\subset \mathcal L (C_2^2 \oplus C_4)$. Since
$\max \Delta^* (C_6)=4$ and $\max \Delta^* (C_2^2 \oplus C_4)=2$ by Proposition \ref{2.2}, it follows that $\mathcal L (C_6) \not\subset \mathcal L (C_2^2 \oplus C_4)$.

\smallskip
\noindent
CASE 4: \, $m = 7$.

Every finite abelian group $G$ with Davenport constant $\mathsf D (G)=7$ is isomorphic to one of the following groups:
\[
C_7, \ C_2 \oplus C_6, \ C_4 \oplus C_4, \ C_2^3 \oplus C_4, \  C_3^3, \ C_2^6 \,.
\]
By Theorem \ref{1.1}.2, $\mathcal L (C_7)$ and $\mathcal L (C_2^6)$ are incomparable in $\Omega_7$. Since all groups $G$ in the above list satisfy $\mathsf D (G)=\mathsf D^* (G)=7$, Theorem  \ref{1.1}.4 implies that $\mathcal L (C_2 \oplus C_6)$ is incomparable in $\Omega_7$. Next we show that $\mathcal L ( C_2^3 \oplus C_4)$ is incomparable in $\Omega_7$. By Theorem \ref{1.1}.3, it is maximal in $\Omega_7$. Thus we have to verify that
\[
\mathcal L (C_3^3) \not\subset \mathcal L ( C_2^3 \oplus C_4) \quad \text{and} \quad \mathcal L (C_4 \oplus C_4) \not\subset \mathcal L ( C_2^3 \oplus C_4) \,.
\]
If $(e_1,e_2,e_3)$ is a basis of $C_3^3$ and $U = e_1^2e_2^2e_3^2(e_1+e_2+e_3)$, then $\mathsf L \big( U (-U) \big) = \{2,3,4,5,7\}$ and  Lemma \ref{5.6} shows  that $\{2,3,4,5,7\} \notin \mathcal L (C_2^3 \oplus C_4)$.

Let $(e_1, e_2)$ be a basis of $C_4 \oplus C_4$ with $\ord (e_1)=\ord (e_2)=4$ and let $U = e_2^3e_1 (e_1+e_2)(e_1+2e_2)^2$. Then
\[
\begin{aligned}
U(-U) & = \Big( e_2^3 (e_1+2e_2)(-e_1-e_2) \Big) \Big( (-e_2)^2(e_1+2e_2)(-e_1) \Big) \big( (-e_2)e_1(e_1+e_2)(-e_1+2e_2)^2 \Big) \\
      & = \Big( (-e_2)e_1(e_1+e_2)(e_1+2e_2)^2 \Big) \Big( e_2(-e_1)(-e_1-e_2)(-e_1+2e_2) \Big) \Big( (e_2(-e_2) \Big)^2 \\
      & = \Big( e_2^2(e_1+2e_2)(-e_1) \Big) \Big( (-e_2)^2(-e_1+2e_2)e_1 \Big) \Big( e_2(-e_2) \Big) \Big( (e_1+e_2)(-e_1-e_2) \Big) \Big( (e_1+2e_2)(-e_1+2e_2) \Big) \\
      & = \Big( e_2(e_1+e_2)(-e_1+2e_2) \Big)   \Big( (-e_2)(-e_1-e_2)(e_1+2e_2) \Big) \Big( e_2(-e_2) \Big)^2 \Big( (e_1+2e_2)(-e_1+2e_2) \Big) \Big( e_1(-e_1) \Big)   \,,
\end{aligned}
\]
whence $\mathsf L \big( U(-U) \big) = [2,7]$. By Lemma \ref{5.6}, we infer that $[2,7] \notin \mathcal L (C_2^3 \oplus C_4)$.

Finally, it remains to show that $\mathcal L (C_4 \oplus C_4)$ and $\mathcal L (C_3^3)$ are incomparable. Since $2 + \max \Delta (C_3^3)  
= 4$ by \cite[Proposition 5.1]{Ge-Gr-Sc11a} and $2 + \max \Delta (C_4 \oplus C_4) = 5$ by \cite[Lemma 3.3]{Ge-Zh15b}, it follows that $\mathcal L (C_4 \oplus C_4) \not\subset \mathcal L (C_3^3)$.
Lemmas \ref{5.4} and \ref{5.5} imply  that $\mathcal L (C_3^3) \not\subset \mathcal L (C_4 \oplus C_4)$.
\end{proof}

\begin{proof}[Proof of Theorem \ref{1.1}.4: Second part]
It remains to show the moreover statement. Let $n \ge 2$ and let $G$ be a finite abelian group with $\mathsf D^* (G) = \mathsf D (G) = 2n+1$ such  that $\mathcal L (G) \subset \mathcal L (C_2 \oplus C_{2n})$. We assert that $G \cong C_2 \oplus C_{2n}$.

If $n = 2$, then $\mathsf D (G)=5$ and the claim follows from Corollary \ref{1.2}. Suppose that $n \ge 3$.
We set $G\cong C_{n_1}\oplus \ldots\oplus C_{n_r}$ with $1<n_1 \t\ldots\t n_r$ and we choose a basis $(e_1, \ldots, e_r)$  of $G$ with $\ord(e_i)=n_i$ for $i \in [1,r]$. By Theorem \ref{1.1}.2, $G$ is neither cyclic nor an elementary $2$-group, whence $r \ge 2$ and $n_r \ge 3$.  We consider the atom $U=e_1^{n_1-1} \cdot \ldots \cdot e_r^{n_r-1}(e_1+\ldots+e_r) \in \mathcal A (G)$ and observe that  $\{2, \mathsf D(G)\} \subset \mathsf L(U(-U)) \in \mathcal L (G) \subset \mathcal L (C_2 \oplus C_{2n})$.
	
If $r\ge 3$, then $U(-U)$ has no minimal zero-sum subsequence of length $3$, which implies that $2n=\mathsf D(G)-1\not\in \mathsf L(U(-U))$, a contradiction to Proposition \ref{5.3}. Thus $r=2$ and we distinguish several cases.

By \cite[Lemma 6.6.4]{Ge-HK06a}, $\mathsf L(U(-U)) = \{2, n_1, n_2, n_1+n_2-2, n_1+n_2-1\} \in \mathcal L (C_{n_1} \oplus C_{n_2})$. Thus, if $n_1$ is odd, then the second largest number of this set is odd, a contradiction to Proposition \ref{5.3}.
	
If  $n_1=2$, then $\mathsf D (G) = n_1+n_2-1=2n+1$ implies that $n_2=2n$, whence $G\cong C_2\oplus C_{2n}$.

If  $n_1\ge 4$  is even and $n_1=n_2$, then  $\rho_3(G)=\mathsf D(G)+\lfloor \frac{\mathsf D(G)}{2}\rfloor>\rho_3(C_2\oplus C_{2n})$ (the first equation follows from \cite[Theorem 6.3.4]{Ge-HK06a} and the inequality follows from \cite[Theorem 5.1]{Ge-Gr-Yu15}), a contradiction to $\mathcal L(G)\subset \mathcal L(C_2 \oplus C_{2n})$.

Suppose that $n_1=4$. We consider the atom $V = e_2^{4n-1}e_1 (e_1+e_2)(e_1-2e_2)(e_1+2e_2)$. Then
\[
V(-V) = \Big( e_2^{4n-1}(e_1+2e_2)(-e_1-e_2) \Big) \Big( (-e_2)^{4n-2}(e_1-2e_2)(-e_1) \Big) \Big( e_1(-e_2)(e_1+e_2)(-e_1+2e_2)(-e_1-2e_2) \Big),
\]
whence $\{2,3,\mathsf D (G)\} \subset \mathsf L \big( V(-V) \big)$, a contradiction to Proposition \ref{5.3}.

Suppose that $n_1 \ge 6$ is even and $n_1 \ne n_2$. 	We set $m= n_2/2$ and consider the atom
\[
V = e_2^{n_2-1} e_1^{n_1-3}(e_1 + m e_2)^2 (e_1+e_2)
\]
and assert that
\[
\mathsf L \big( V(-V) \big) = \{2, n_1+m-2, n_1+m-1,n_1+m,n_2, n_1+n_2-3,n_1+n_2-2,n_1+n_2-1\} \,.
\]
Clearly, $\{2, \mathsf D (G)\} \subset \mathsf L \big( V(-V) \big)$. If $W = (e_1+e_2)(-e_1)(-e_2)$, then $W(-W)$ divides $V(-V)$ and gives rise to a factorization of length $\mathsf D (G)-1$. Now we consider a factorization $z \in \mathsf Z \big( V(-V) \big)$ with $|z| \notin \{2, \mathsf D (G)-1, \mathsf D (G)\}$. Then there is an atom $W \in \mathcal A (G)$ with $W$ divides $z$, $|W| \ge 3$, and with $(e_1+me_2) \mid W$.

If $W = (-e_2)e_1^{n_1-3}(e_1 + m e_2)^2 (e_1+e_2)$, then $W(-W)$ divides $V(-V)$ and gives rise to a factorization of length $n_2$. If $W = (e_1 + m e_2)^2 (-e_1)^2$, then
\[
V(-V) = W(-W) \big( (e_1+e_2)(-e_1-e_2) \big) \big( e_1 (-e_1) \big)^{n_1-5} \big( e_2(-e_2) \big)^{n_2-1}
\]
is a factorization of length $n_1+n_2-3$. If $W = (e_1+me_2)(-e_1-e_2)(-e_2)^{m-1}$, then
\[
V(-V) = W(-W) \big( (e_1+me_2)(-e_1+me_2) \big) \big( e_2(-e_2) \big)^{m} \big( e_1(-e_1) \big)^{n_1-3}
\]
is a factorization of length $n_1+m$. If $W = (e_1+me_2)(-e_1)e_2^m$, then
\[
V(-V) = W(-W) \big( (e_1+me_2)(-e_1+me_2) \big) \big( (e_1+e_2)(-e_1-e_2) \big) \big( e_1(-e_1) \big)^{n_1-4} \big( e_2(-e_2) \big)^{m-1}
\]
is a factorization of length $n_1+m-1$. If $W = (e_1+me_2)(-e_1)(-e_2)^m$, then we obtain again a factorization of length $n_1+m-1$. If $W = (e_1+me_2)(-e_1-e_2)e_2^{m+1}$, then
\[
V(-V) = W(-W) \big( (e_1+me_2)(-e_1+me_2) \big) \big( e_1(-e_1) \big)^{n_1-3} \big( e_2(-e_2) \big)^{m-2}
\]
is a factorization of length $n_1+m-2$.

Since $n_1 \ge 6$, $\mathsf L \big( V(-V) \big) \setminus \{2\}$ consists of seven elements but it is not an interval. Thus Proposition \ref{5.3} implies that $\mathsf L \big( V(-V) \big) \notin \mathcal L (C_2 \oplus C_{2n})$, a contradiction to $\mathcal L(G)\subset \mathcal L(C_2 \oplus C_{2n})$.
\end{proof}

\noindent
{\bf Acknowledgement.} We thank Qinghai Zhong for many helpful discussions and the anonymous reviewers for their careful reading.

\providecommand{\bysame}{\leavevmode\hbox to3em{\hrulefill}\thinspace}
\providecommand{\MR}{\relax\ifhmode\unskip\space\fi MR }
\providecommand{\MRhref}[2]{%
  \href{http://www.ams.org/mathscinet-getitem?mr=#1}{#2}
}
\providecommand{\href}[2]{#2}

\end{document}